\documentclass[aap]{imsart}

\RequirePackage{amsthm,amsmath,amsfonts,amssymb}
\RequirePackage[numbers,sort&compress]{natbib}
\RequirePackage[colorlinks,citecolor=blue,urlcolor=blue]{hyperref}
\RequirePackage{graphicx}

\startlocaldefs
\theoremstyle{plain}

\newtheorem{proposition}{Proposition}[section]
\newtheorem{corollary}{Corollary}[section]
\newtheorem{lemma}{Lemma}[section]
\theoremstyle{definition}

\newtheorem*{example}{Example}

\endlocaldefs

\usepackage{enumitem}
\usepackage{dsfont}
\usepackage{mathtools}
\usepackage{CoOL-bib/customCommands}

\begin{document}

\begin{frontmatter}
\title{Bounds on Spectral Gaps for Non-Reversible Markov Chains with Applications to Temporal Difference Learning}
\runtitle{Spectral Gaps for Non-Reversible Markov Chains}

\begin{aug}
\author[A]{\fnms{Andrew}~\snm{Lamperski} 
\ead[label=e1]{alampers@umn.edu}}
\address[A]{Electrical and Computer Engineering,
University of Minnesota \printead[presep={ ,\ }]{e1}}

\end{aug}

\begin{abstract}
This work is motivated by the analysis of temporal difference algorithms, where stability can be guaranteed by bounding the eigenvalues of an associated matrix derived from a, typically non-reversible, Markov kernel. We generalize the existing sufficient conditions for stability and show that the associated eigenvalues can be bounded in terms of the Dirichlet spectral gap of the Markov kernel. We derive a collection of methods for showing that non-reversible Markov chains have positive spectral gaps. We show that if a Markov chain has positive absolute spectral gap, then it has a positive Dirichlet spectral gap.  In the case of discrete-time linear Gaussian systems, we give explicit bounds for both Dirichlet and absolute spectral gaps. Additionally, we present an example of a Markov chain which is $V$-uniformly geometrically ergodic but has zero Dirichlet spectral gap. 
\end{abstract}



\end{frontmatter}

\section{Introduction}
Spectral properties of operators induced by Markov kernels are commonly studied as a means to analyze convergence to stationary distributions \cite{chen2006eigenvalues,douc2018markov,bakry2013analysis}. Spectral gaps, in particular, can be used to quantify convergence speed. 

When the Markov chain is reversible, spectral gaps can be expressed in terms of an associated Dirichlet form. However, when the Markov chain is non-reversible, the Dirichlet form and associated spectral gap can be defined \cite{chung2005laplacians}, but their connection to the spectrum of the original operator and convergence properties is less obvious. Related notions, such as absolute spectral gaps, can give more precise information about convergence for non-reversible Markov chains \cite{kontoyiannis2012geometric}.   

While other spectral gap notions may be more useful for analyzing convergence for non-reversible Markov chains, the spectral gap definition via Dirichlet forms naturally arises in the analysis of reinforcement learning methods \cite{liu2021temporal,tian2026bridging}. We show concretely how the eigenvalues of a matrix arising in the analysis of temporal difference learning methods, \cite{konda2003onactor,tsitsiklis1999average}, can be bounded in terms of the Dirichlet spectral gap for (typically non-reversible) Markov chain.  Bounds on the eigenvalues of this matrix, in turn, utilized for quantitative convergence guarantees for actor-critic reinforcement learning.  Often, however, the bounds are assumed rather than proved \cite{wu2020finite,xu2020improving,ganesh2025sharper}. The connection to reinforcement learning motivates the main results on this paper, which gives a collection of methods for proving that the Dirichlet 
spectral gap is positive and giving positive lower bounds.  

The main contributions of the work are as follows. We generalize existing conditions for showing that eigenvalues for matrices arising in temporal difference learning have negative real parts. Existing conditions require that the underlying Markov chain is $V$-uniformly geometrically ergodic. Our results show that irreducibility suffices. We show, further, that the real-parts of the eigenvalues of the corresponding matrices can be bounded in terms of the (Dirichlet) spectral gap of the corresponding Markov operator. We give general conditions for positive spectral gaps based on density functions of the Markov kernel. Then we show how to derive quantitative bounds on spectral gaps via comparison methods. In particular, we show how uniform geometric ergodicity and absolute spectral gaps can lead to bounds on Dirichlet spectral gaps. We analyze the particular case of linear Gaussian systems, as it commonly arises in control problems, and give explicit bounds on Dirichlet and absolute spectral gaps in terms of the original system matrices. Finally, we construct an example of a $V$-uniformly geometrically ergodic Markov chain which has zero Dirichlet and zero absolute spectral gap, thus showing that $V$-uniform geometric ergodicity is not sufficient to imply a positive Dirichlet spectral gap.

The paper is organized as follows. Section~\ref{sec:background} gives background on Markov operators and defines the spectral gap properties studied in the paper. Section~\ref{sec:TD} describes the connection with spectral gaps and temporal difference learning. 
Sufficient conditions for positive spectral gaps are given in Section~\ref{sec:gaps}. The example of a $V$-uniformly geometrically ergodic Markov chain with zero Dirichlet spectral gap is given in Section~\ref{sec:noGap}. Conclusions are given in Section~\ref{sec:conclusion}.

\paragraph*{Notation}
$\bbR$ and $\bbC$ denote the set of real and complex numbers, respectively. If $M$  is a matrix, its transpose is denoted by $M^\top$. Random variables are denoted as bold symbols, e.g. $\bx$. The expected value of a random variable is denoted by $\bbE[\bx]$.

\section{Markov Operators and Spectral Gaps}
\label{sec:background}
In this section, $P$ denotes a Markov kernel with invariant probability measure $\mu$ over measure space $(\cX,\cB)$, where $\cX$ is the state space and $\cB$ is a corresponding $\sigma$-algebra over $\cX$.

We review some required concepts from the operator-theoretic interpretation of Markov chains. For more details, see \cite{douc2018markov}.

Let $L^2(\mu)$ denote the set of square integrable functions $f:\cX\to\bbC$ with inner product defined by
\begin{equation}
  \label{eq:innerProduct}
\langle f,g\rangle = \int_{\cX}\bar f(x) g(x) \mu(dx). 
\end{equation}

The Markov kernel, $P$, can be viewed as an operator over $L^2(\mu)$ with action defined by
$$
Pf(x)=\int_{\cX}P(x,dy)f(y).
$$

The measure, $\mu$, can be viewed as a linear functional over $L^2(\mu)$ with action defined by
$$
\mu f=\int_{\cX}\mu(dx)f(x). 
$$

The adjoint of $P$ with respect to the inner product from (\ref{eq:innerProduct}) is given by
$$
P^\star f(x) = \frac{d\mu_f}{d\mu}(x),
$$
where $\mu_f$ is the complex measure defined by
$$
\mu_f(A)=\int_{\cX}\mu(dx)f(x)P(x,A).
$$
for all $A\in\cB$.

The operator, $P^\star$, can also be interpreted as a Markov kernel defined by
$$
P^\star(x,A)=P^{\star}\indic_A(x),
$$
where $\indic_A$ is the indicator function of the set $A$:
$$
\indic_{A}(x)=\begin{cases}1 & x\in A \\
  0 & x \notin A.
  \end{cases}
  $$
Furthermore, $\mu$ is an invariant measure for $P^\star$.

We say that $P$ is \emph{reversible} with respect to $\mu$  if $P=P^\star$. Using the terminology of \cite{fill1991eigenvalue}, we call $P_a:=\frac{1}{2}(P+P^\star)$  the \emph{additive reversibilization} of $P$.

Using the terminology of \cite{saloff1997lectures}, the \emph{Dirchlet form} associated with $P$ and $\mu$ is defined by:
\begin{equation}
  \label{eq:dirichlet}
  \cE(f,g)=\Re\left(\langle f,(I-P)g\rangle\right),
\end{equation}
where $\Re$ denotes the real-part. 
(We use a slightly different definition from \cite{saloff1997lectures} for consistency with our definition of the inner product.)

In particular, direct calculations show that:
\begin{equation}
  \label{eq:dirichletAlt}
  \cE(f,f)=\frac{1}{2}\int_{\cX}\int_{\cX}\mu(dx)P(x,dy)|f(x)-f(y)|^2=\left\langle f, \left(I-\frac{1}{2}(P+P^\star)\right)f\right\rangle
\end{equation}

Let $\var(f)=\int_{\cX}\mu(dx)|f(x)|^2-\left|\int_{\cX}\mu(dx)f(x) \right|^2$ denote the variance of $f$. In the terminology of \cite{saloff1997lectures,chen2006eigenvalues}, we define the \emph{spectral gap} of $P$ by
\begin{equation}
  \label{eq:gap}
\mathrm{gap}(P) = \inf\left\{
  \frac{\cE(f,f)}{\var(f)}\middle|0<\var(f)<\infty
  \right\}.
\end{equation}

This notion of spectral gap goes by various names such as the \emph{optimal Poincar\'e constant} \cite{dyer2006markov} or the \emph{first Dirichlet eigenvalue} \cite{chen2006eigenvalues}. 

We will discuss in Section~\ref{sec:TD} how the definition of spectral gap from \eqref{eq:gap} is well-suited for the analysis of matrices arising in temporal difference methods. 

A common alternative spectral gap notion is the \emph{absolute spectral gap}, which is examined in \cite{douc2018markov,kontoyiannis2012geometric}. To define the absolute spectral gap, let $L_0^2(\mu)$ denote the subspace of $L^2(\mu)$ of functions, $f$, with $\mu f=0$, and let $\mathrm{spec}(P|L_0^2(\mu))$ be the spectrum of $P$ restricted to the subspace $L_0^2(\mu)$. Then the absolute spectral gap of  $P$ is defined by:
$$
\absgap(P)=1-\sup\{|\lambda| | \lambda \in\mathrm{spec}(P|L_0^2(\mu))\}.
$$
Lemma~\ref{lem:absGap2Gap} in Subsection~\ref{ss:comparison} shows that $\absgap(P)>0$ implies that $\gap(P)>0$. Example~\ref{ex:noGap2AbsGap} below in this section shows that the converse does not hold.

Other spectral gap notions include the pseudo-spectral gap \cite{paulin2015concentration,wolfer2019estimating} and the iterated Poincar\'e gap \cite{chatterjee2025spectral,huang2025non}.

We say that $P$ admits a kernel density with respect to $\mu$ if there is a function, $K:\cX\times \cX\to \bbR$, such that for all $A\in\cB$:
$$
P(x,A)=\int_{A}\mu(dy)K(x,y).
$$

When $P$ admits a kernel density with respect to $\mu$, the actions of $P$ and $P^\star$ can be expressed as:
\begin{align*}  
  Pf(x)&=\int_{\cX}\mu(dy)K(x,y)f(y) \\
  P^\star f(x)&=\int_{\cX}\mu(dy)K(y,x)f(y).
\end{align*}
In this case, $P$ is reversible if $K(x,y)=K(y,x)$ for all $x,y\in\cX$. 

When $\cX=\{1,\ldots,n\}$ and $\cB = 2^{\cX}$, we identify $P$ with the corresponding $n\times n$ stochastic matrix and identify $\mu$ with the corresponding $1\times n$ vector. If $\mu_i >0$ for $i=1,\ldots,n$, then $P$ admits a kernel density with respect to $\mu$, which is given by $K_{ij}=P_{ij}/\mu_j$, and $P^\star$ corresponds to the matrix with $P^\star_{ij}=K_{ji}\mu_j=\frac{1}{\mu_i}P_{ji}\mu_j$. Setting $D=\diag(\mu)$, the expressions for $K$ and $P^\star$ can be expressed in matrix form as:
\begin{align*}
  K&= PD^{-1}\\
  P^\star &= D^{-1} P^\top D
\end{align*}
In this case, $P$ is reversible if and only if $P=D^{-1}P^\top D$. 

\begin{example}
  \label{ex:noGap2AbsGap}
  This example gives a finite-state Markov chain with $\absgap(P)=0$ and $\gap(P)=2>0$. Thus, in contrast to Lemma~\ref{lem:absGap2Gap}, $\gap(P)>0$ does not imply that $\absgap(P)>0$. Let 
  $$
  P=\begin{bmatrix}
  0 & 1 \\
  1 & 0
  \end{bmatrix}.
  $$
  This defines an irreducible, reversible (but periodic) Markov chain, and the unique invariant distribution is 
  $$
  \mu=\begin{bmatrix}\frac{1}{2} & \frac{1}{2}\end{bmatrix}.
  $$

  Direct calculations give:
  \begin{gather*}
  \mathrm{spec}(P|L_0^2(\mu)) = \{-1\} \implies \absgap(P)=0 \\
  \gap(P)=2.
  \end{gather*}
  In this example, the difference between spectral and absolute spectral gaps reflects a basic difference in their definitions. Namely, when restricted to reversible operators, the (Dirichlet) spectral gap measures the distance of the spectrum from $1$ (aside from the eigenvalue at exactly $1$), where the absolute gap measures the deviation of the spectrum from all of the unit circle. 

  For Markov chain matrices which are both reversible and aperiodic, only a single eigenvalue (at $1$) could occur on the unit circle and both the (Dirichlet) spectral gap and the absolute spectral gap would be positive.  
\end{example}

An elementary lemma needed in the proofs is:

\begin{lemma}
  \label{lem:PrIrreducible}
  Let $P$ be a Markov kernel with invariant probability measure $\mu$. 
    If $P$ is irreducible, then $P_a:=\frac{1}{2}(P+P^\star)$ must be irreducible as well.
\end{lemma}
\begin{proof}
Recall that $P$ is irreducible if the state space has an accessible small set, which we denote by $C\in\cB$. By accessible, we mean that for all $x\in\cX$, there is a number $n\ge 1$ such that $P^n(x,C) >0$. But then $C$ must also be accessible for $P_a$, since $P_a^n(x,C)\ge \frac{1}{2^n}P^n(x,C)$. By small, we mean that there is a non-zero non-negative measure, $\nu$, and a positive integer $m$ such that for all $x\in C$ and all $B\in\cB$ we have $P^m(x,B)\ge \nu(B)$. (Theorem 9.2.15 of \cite{douc2018markov} shows that we can take $\nu(B)=\delta \mu(B)$ for some $\delta >0$.) Then, as before, we have $P_a^m(x,B)\ge \frac{1}{2^m} P^m(x,B)\ge \frac{1}{2^m}\nu(B)$. So, $C$ is small for $P_a$ as well.
\end{proof}

The following elementary lemma relates the definition of the spectral gap from \eqref{eq:gap} to the spectrum of the additive reversibilization.
\begin{lemma}
  \label{lem:gapFromSpectrum}
  If $P$ is irreducible, then
  $\gap(P)=1-\sup\left(\spec(P_a)\setminus\{1\}\right)$
\end{lemma}
\begin{proof}
  The definition from \eqref{eq:gap} can be expressed equivalently as 
  $$
  \gap(P)=\inf\{\cE(f,f)| f\in L_0^2(\mu),\: \|f\|=1\}=1-\sup\{\langle f,P_a f\rangle|f\in L_0^2(\mu),\: \|f\|=1\}.
  $$
  A standard result from operator theory, see e.g. Theorem 22.A.19 of \cite{douc2018markov}, implies that 
  $$
  \sup\{\langle f,P_a f\rangle|f\in L_0^2(\mu),\: \|f\|=1\}=\sup\:\spec(P_a|L_0^2(\mu)).
  $$ 
  Since $P_a$ is irreducible, Proposition 2.1.2 of \cite{douc2018markov} implies that $P_a$ has an eigenvalue of multiplicity $1$ at $\lambda=1$, and the corresponding eigenvectors are constants. The spectral theorem for bounded self-adjoint operators implies that $$
  \spec(P_a)\setminus\{\ones\} = \spec(P_a|(\mathrm{span}\{\ones\})^{\perp})=\spec(P_a|L_0^2(\mu)),
  $$
  where the second equality follows because $L_0^2(\mu)$ is the orthogonal complement of the space of constant functions. The result now follows.
\end{proof}

\section{Motivation from Temporal Difference Policy Evaluation}
\label{sec:TD}

Temporal difference (TD) policy evaluation is one of the main components of actor-critic algorithms, a common family of algorithms from reinforcement learning \cite{sutton2018reinforcement}. The main results of this section generalize existing sufficient conditions for the asymptotic approximation of the policy evaluation method $TD(\lambda)$ to be stable. They also show how bounds on the corresponding matrix can be given in terms of spectral gaps.  

Let $c:\cX\to\bbR$ be a function of the state space and let $\bx_k$ be a sequence of states generated by the Markov kernel, $P$. Let $\mu_k$ be the distribution of $\bx_k$. In the typical setting of average-cost reinforcement learning, it is assumed that there is a unique invariant distribution, $\mu$, such that $\mu_k$ converges to $\mu$. 

Let $\bar c= \mu c$. The \emph{marginal value function} is given by: 
$$
V(x)=\bbE\left[ \sum_{k=0}^{\infty}\left(c(\bx_k)-\bar c\right)\middle|\bx_0=x\right]
$$
and satisfies:
$$
V=c-\ones \bar c + PV
$$
where $\ones$ is the constant function taking the value of $1$.

The \emph{policy evaluation} problem involves estimating $V$, and temporal difference methods estimate $V$ directly from observations of the Markov chain. When the state space, $\cX$, is large or infinite, it is often impossible to compute $V$ exactly. In the most basic setting, we fix a collection of ``features", $\phi_i:\cX\to\bbR$ and search for parameters $\theta^i$ such that:
$$
V(x)\approx \sum_{i=1}^m\phi_i(x)\theta^i =:\phi(x)^\top \theta.
$$

A common temporal difference algorithm, known as $\mathrm{TD}(\lambda)$ is given by:
\begin{align*}
  \bJ_{k+1} &= \bJ_k +\alpha_k(c(\bx_k)-\bJ_k) \\
  \btheta_{k+1} &=\btheta_{k}+\alpha_k\left(c(\bx_k)-\bJ_k+\phi(\bx_{k+1})^\top \btheta_k-\phi(\bx_k)^\top\btheta_k \right)\bz_k \\
  \bz_{k+1}&=\lambda \bz_k+\phi(\bx_{k+1})
\end{align*}
where $\bJ_k$ is an estimate of $\bar c$, $\alpha_k>0$ is a sequence of step sizes, and $\lambda \in [0,1]$.

We will focus on the case that $\lambda \in [0,1)$. The case of $\lambda=1$ requires alternate techniques outside the scope of the paper. See~\cite{konda2003onactor}.

As discussed in \cite{tsitsiklis1999average}, under standard stochastic approximation step size assumptions, $\sum_{k=0}^{\infty}\alpha_k=\infty$, $\sum_{k=0}^{\infty}\alpha_k^2<\infty$, and $\lambda \in [0,1)$, the values of $(\bJ_k,\btheta_k)$ from $\mathrm{TD}(\lambda)$ behave, asymptotically, like solutions to the differential equation:
\begin{align*}
  \frac{dJ_t}{dt} &= \bar c-J_t \\
  \frac{d\theta_t}{dt} &= b -g J_t +A\theta_t, 
\end{align*}
where
\begin{align*}
  \phi(x)&=\begin{bmatrix}\phi_1(x) & \cdots & \phi_n(x)\end{bmatrix}^\top \\
  P^{(\lambda)} &= (1-\lambda)\sum_{i=0}^{\infty}\lambda^i P^{i+1}\\
  b_i&=\langle \phi_i,c\rangle \\
  g_i&=\mu \phi_i\\
  A_{ij} &=\langle \phi_i,(P^{(\lambda)}-I) \phi_j\rangle.  
\end{align*}

Note that in the case that $\lambda=0$, the algorithm simplifies, since we can take $\bz_k=\phi(\bx_k)$. Furthermore, in this case, $P^{(0)}=P$.

If $A$ is Hurwitz (i.e. all its eigenvalues have negative real parts), then the differential equation has a unique, asymptotically stable equilibrium at $\begin{bmatrix}\bar c \\ A^{-1}(g\bar c-b)\end{bmatrix}$. In this case, $\begin{bmatrix}\bJ_k \\\btheta_k\end{bmatrix}$ will converge to this equilibrium. 

A variety of papers have derived sufficient conditions for $A$ to be Hurwitz under different assumptions. The case of finite state spaces and $\lambda \in [0,1)$ is examined in \cite{tsitsiklis1999average}, the case that $\cX$ is a Polish space, $P$ is $V$-uniformly geometrically ergodic, and $\lambda \in (0,1]$ is examined in \cite{konda2003onactor}.

The following lemma shows that when $\lambda\in [0,1)$, irreducibility and spectral gap properties of $P^{(\lambda)}$ can be reduced to irreducibility and spectral gap properties of $P$. It will be used to give sufficient conditions for $A$ to be Hurwitz.

\begin{lemma}
  \label{lem:sampledReduction}
Let $P$ be a Markov kernel and $\lambda \in [0,1)$. The following properties hold.
\begin{enumerate}
  \item \label{it:sampledMarkov} $P^{(\lambda)}$ is a Markov kernel. 
  \item If $\mu$ is an invariant probability distribution for $P$, then $\mu$ is an invariant probability distribution for $P^{(\lambda)}$.
  \item \label{it:sampledIrreducible} If $P$ is irreducible, then so is $P^{(\lambda)}$.
  \item 
    \label {it:sampledDirichlet} If $\cE_{P}$ and $\cE_{P^{(\lambda)}}$ are the Dirichlet forms of $P$ and $P^{(\lambda)}$ with respect to an invariant probability distribution, $\mu$, then $\cE_{P^{(\lambda)}}(f,f)\ge(1-\lambda)\cE_P(f,f)$. 
  \item  \label{it:sampledGap} $\gap(P^{(\lambda)})\ge (1-\lambda) \gap(P)$.
\end{enumerate}
\end{lemma}

\begin{proof}
  Note that for $\lambda\in [0,1)$, $P^{(\lambda)}$ is a specific type of \emph{sampled chain}, as examined in  \cite{meyn2012markov,douc2018markov}. In this case,
  Properties~\ref{it:sampledMarkov} -- \ref{it:sampledIrreducible} are standard properties of sampled chains, and so the proofs are omitted. When $\lambda = 0$, $P^{(0)}=P$, and so the all of the properties are immediate. 

  Only properties \ref{it:sampledDirichlet} and  \ref{it:sampledGap} are proved.

  Say that $P$ is a Markov kernel with invariant probability distribution $\mu$.
  Note that 
  $$
  P^{(\lambda)} = (1-\lambda) P + \lambda P P^{(\lambda)},
  $$
  where $P P^{(\lambda)}$ is also a Markov kernel with invariant probability distribution $\mu$.  
  
  Let $\cE_P$, $\cE_{P^{(\lambda)}}$, and $\cE_{P P^{(\lambda)}}$  be the Dirichlet forms of $P$, $P^{(\lambda)}$, and $P P^{(\lambda)}$, respectively.
  Using~\eqref{eq:dirichletAlt} shows that:
  \begin{align*}
    \cE_{P^{(\lambda)}}(f,f) &= (1-\lambda)\cE_P(f,f)+\lambda \cE_{PP^{(\lambda)}}(f,f) \\
                             &\ge (1-\lambda)\cE_{P}(f,f).
  \end{align*}

  When $\var(f)>0$, dividing both sides by $\var(f)$ and minimizing over $f$ gives the result. 
\end{proof}

Proposition~\ref{prop:irreducibleHurwitz}, below, substantially generalizes the existing sufficient conditions for $A$ to be Hurwitz.  It shows that $A$ will be Hurwitz as long as $P$ is irreducible and common conditions on the features, $\phi$, hold. Proposition~\ref{prop:gapToHurwitz} shows that explicit bounds on the real parts of the eigenvalues (and an explicit Lyapunov inequality) can be given when $\gap(P)>0$.

\begin{proposition}
  \label{prop:irreducibleHurwitz}
  Let $P$ be a Markov kernel with invariant probability distribution $\mu$. 
   Let $\phi_i\in L^2(\mu)$ for $i=1,\ldots,n$ and define the matrices $A$ and $S\in\bbC^{n\times n}$ by:
  \begin{align*}
    A_{ij}&=\langle \phi_i,(P^{(\lambda)}-I)\phi_j\rangle.
  \end{align*}

  If the following conditions hold:
  \begin{enumerate}
  \item \label{it:irreducible} $P$ is irreducible,
  \item  \label{it:li}
    the functions $\{\phi_1,\ldots,\phi_n\}$ are linearly independent,
  \item  \label{it:const}
    the constant function, $\ones$, is not in the span of $\{\phi_1,\ldots,\phi_n\}$,
  \end{enumerate}
  then $A+A^\star$ is negative definite. In particular, $A$ is Hurwitz.
\end{proposition}

The conditions on $\phi_i$ are common. See \cite{tsitsiklis1999average}.

\begin{proof}
  Let  $v\in\bbC^n$ and set $w\in L^2(\mu)$ by $w=\sum_{i=1}^n \phi_i v_i$.
  Let $v^\star$ and $A^\star$ denote the respective Hermitian conjugates. 
  Then a  direct calculation gives 
  \begin{align}
    \nonumber
    v^\star \left(\frac{1}{2}A+\frac{1}{2}A^\star \right)v
    &=\frac{1}{2}v^\star Av+\frac{1}{2}\overline{v^\star A v} 
    \\
    \nonumber
    & =\frac{1}{2}\langle w,(P^{(\lambda)}-I)w\rangle + \frac{1}{2}\langle (P^{(\lambda)}-I)w,w\rangle \\
    \nonumber
    &=-\cE_{P^{(\lambda)}}(w,w)\\
  \label{eq:matrixToDirichlet}
    &\le -(1-\lambda) \cE_P(w,w)\\
\nonumber
    &\le 0.
  \end{align}
  Here, the first inequality follows from (\ref{eq:dirichletAlt}), while the second is from Property \ref{it:sampledDirichlet} of Lemma~\ref{lem:sampledReduction}. 

  So, $\frac{1}{2}(A+A^\star)$ is negative semidefinite. Then we will have that $A$ must be Hurwitz if we can show that $\frac{1}{2}(A+A^\star)$ has no eigenvalue at $0$.

  Let $w=\sum_{i=1}^n\phi_i v_i$ and assume that $\frac{1}{2}v^\star(A+A^\star)v=0$, which, according to (\ref{eq:matrixToDirichlet}), is equivalent to having $\cE_P(w,w)=0$. In the rest of the proof, we will denote $\cE_P=\cE$ for simpler notation. 

  We will show that $w$ must be a constant $\mu$-almost everywhere. But then Condition~\ref{it:const} implies that the only constant function in the span of $\{\phi_1,\ldots,\phi_n\}$ is $0$. Then linear independence would imply that $v=0$. So, we must have $\frac{1}{2}v^\star(A+A^\star)v<0$ if $v\ne 0$.

  Now we turn to proving  that $w$ must be constant if $\cE(w,w)=0$.

  Let $P_a = \frac{1}{2}(P+P^\star)$ be the additive reversiblization. According to (\ref{eq:dirichletAlt}), we  have
  $$
  \cE(w,w)=\langle w,(I-P_a)w\rangle. 
  $$

  Lemma~\ref{lem:PrIrreducible} shows that $P_a$ is irreducible since $P$ is.
Thus, we have that $1$ is an eigenvalue of $P_a$ with multiplicity $1$, and every corresponding eigenvector must be constant $\mu$-almost everywhere. See Proposition 22.1.2 of \cite{douc2018markov}.

  Furthermore, $P_a$ is self-adjoint with its spectrum satisfying $\spec(P_a)\subset [-1,1]$. It follows from the spectral theorem that the Dirichlet form can be written as:
  $$
  \cE(w,w)=\langle w,(I-P_a)w\rangle = \int_{\spec(P_a)}(1-\lambda)\langle w,E(d\lambda) w\rangle, 
  $$
  where $E$ is a uniquely defined resolution of the identity. In particular, $\langle w,E(\cdot) w\rangle$ is a non-negative measure over the Borel subsets of $\spec(P_a)$.

  We will show that $\langle w, E(\spec(P_a)\setminus\{1\}) w\rangle =0$. To this end, let $S_n = \spec(P_a)\cap \left[-1,1-\frac{1}{n} \right]$. Note that $\spec(P_a)\setminus \{1\} = \bigcup_{n=1}^{\infty} S_n$. Furthermore, $\cE(w,w)=0$ implies that
  \begin{align*}
    0&=
    \int_{\spec(P_a)}(1-\lambda) \langle w, E(d\lambda) w\rangle\\
    &\ge
      \int_{S_n} (1-\lambda) \langle w, E(d\lambda) w\rangle \\
    &\ge \frac{1}{n} \langle w, E(S_n) w\rangle
  \end{align*}
  And so, $\langle w,E(S_n) w\rangle =0$ for all $n\ge 1$, which implies that $\langle w,E(\spec(P_a)\setminus \{1\})w\rangle=0$.

  Now, since $E(\spec(P_a)\setminus\{1\})$ is a self-adjoint projection, we have that
  $$
\langle w, E(\spec(P_a)\setminus\{1\}) w\rangle = 
\langle  E(\spec(P_a)\setminus\{1\}) w, E(\spec(P_a)\setminus\{1\}) w\rangle = 0,
$$
which implies that $E(\spec(P_a)\setminus\{1\})w=0$ (in an $L^2(\mu)$ sense).

Furthermore, since $E$ is a resolution of the identity, we have (in an $L^2(\mu)$ sense):
\begin{align*}
  w&=\left(E(\spec(P_a)\setminus\{1\})+E(\{1\})\right)w \\
   &=E(\{1\})w,
\end{align*}
which implies that $w$ is in the range space of $E(\{1\})$. Theorem 12.29 of \cite{rudin1991functional} implies that the range space of $E(\{1\})$ is equal to the null space of $P_a-I$. In particular, $w$ must be in the eigenspace of $P_a$ corresponding to the eigenvalue $1$. So, as discussed above, $w$ must be constant $\mu$-almost everywhere. Thus, Conditions~\ref{it:li} and \ref{it:const} imply that $v$ must be $0$, and $\frac{1}{2}(A+A^\star)$ is negative definite. 
\end{proof}

\begin{proposition}
  \label{prop:gapToHurwitz}
  Let $P$ be a Markov kernel with invariant probability distribution $\mu$. Let $\phi_1,\ldots,\phi_n$ and $A$ be as in Proposition~\ref{prop:irreducibleHurwitz} and $\lambda\in [0,1)$.
  Define the matrix $S\in\bbC^{n\times n}$ by
  \begin{align*}
  S_{ij}&=\langle \phi_i,\phi_j\rangle - \overline{(\mu \phi_i)}(\mu\phi_j).        
  \end{align*}
  
  Let $\preceq$ denote the positive semidefinite partial order, and let $\gamma$ denote the smallest eigenvalue of $S$. Then $A$ satisfies the linear Lyapunov inequality:
  \begin{equation}
    \label{eq:linLyap}
  \frac{1}{2}(A+A^\star)\preceq -(1-\lambda)\gap(P)\gamma I.
\end{equation}
In particular, the eigenvalues of $A$ have real parts at most $-(1-\lambda)\gap(P)\gamma$.

  If the following conditions hold:
  \begin{enumerate}
  \item \label{it:gap} $\gap(P)>0$ 
  \item  \label{it:li2}
    the functions $\{\phi_1,\ldots,\phi_n\}$ are linearly independent,
  \item  \label{it:const2}
    the constant function, $\ones$, is not in the span of $\{\phi_1,\ldots,\phi_n\}$, 
  \end{enumerate}
 then $(1-\lambda)\gap(P)\gamma >0$ and so $A$ is Hurwitz.
  \end{proposition}
 \begin{proof}

   Let  $v\in\bbC^n$ and set $w\in L^2(\mu)$ by $w=\sum_{i=1}^n \phi_i v_i$. Then
   using (\ref{eq:matrixToDirichlet}), followed by the definition of the  and then the definition of $S$ gives 
  \begin{align}
    \nonumber
    \frac{1}{2}v^\star(A+A^\star)v
    &\le -(1-\lambda)\cE_P(w,w)\\
    \nonumber
    &\le -(1-\lambda)\mathrm{gap}(P)\var(w)\\
    \nonumber
    &= -(1-\lambda)\mathrm{gap}(P)v^\star S v \\
    \nonumber
    &\le -(1-\lambda)\gap(P) \gamma \|v\|_2^2,
  \end{align}
  where $\|\cdot\|_2$ denotes the Euclidean norm. 

  Thus, (\ref{eq:linLyap}) holds. The bound on the eigenvalues follows by taking $v$ to be an eigenvector of $A$ of unit norm. 

  Now, assume that $\{\phi_1,\ldots,\phi_n\}$ are linearly independent, and $\ones$ is not in their span. If $v\ne 0$, then $w$ cannot be $\mu$-almost everywhere constant.  It follows that $\var(w)=v^\star S v>0$, so that $S$ is positive definite. In particular, $\gamma >0$.

  If, furthermore, $\gap(P)>0$, the bound satisfies $-(1-\lambda)\gap(P)\gamma <0$. 
\end{proof}

A subtle condition can arise in which $\gap(P)=0$, but $P$ is irreducible. In this case, by Proposition~\ref{prop:irreducibleHurwitz} we can show that $A$ will be Hurwitz, but we cannot get a non-trivial bound on its eigenvalues via the spectral gap. 

In Section~\ref{sec:noGap}, we construct a Markov chain which is irreducible and $V$-uniformly geometrically ergodic, with $\gap(P)=\absgap(P)=0$. In this case, Proposition~\ref{prop:irreducibleHurwitz} shows that $A$ will be Hurwitz for $\lambda \in [0,1)$, and the earlier results of \cite{konda2003onactor} show that $A$ will be Hurwitz when $\lambda \in (0,1)$. But, again, no non-trivial  bound on the eigenvalues via spectral gaps can be given.

\section{Sufficient Conditions for Positive Spectral Gaps}
\label{sec:gaps}

Proposition~\ref{prop:gapToHurwitz} shows that a key condition to get $A$ Hurwitz is that the Markov chain has non-zero spectral gap. In this section, we describe a collection of sufficient conditions that ensure a positive spectral gap.
Subsection~\ref{ss:kernel} shows that $P$ has a positive gap if it admits a kernel which is sufficiently well-behaved. This result, in particular, suffices to prove that irreducible Markov chains over finite state spaces have positive spectral gaps.  Subsection~\ref{ss:comparison} shows how to derive bounds on spectral gaps by comparison of one Markov operator with another. In particular, the results show that uniformly geometrically ergodic Markov chains have positive spectral gaps. Subsection~\ref{ss:linGauss} gives bounds on spectral gaps in the commonly studied case of linear Gaussian systems.

\subsection{Kernel Conditions}
\label{ss:kernel}
The following proposition gives a sufficient condition for a positive spectral gap which is often satisfied in practice. Below, $\mu\otimes\mu$ denotes the product measure of $\mu$ over $\cX\times\cX$. In particular, it will be satisfied for Markov chains on finite state spaces.

\begin{lemma}
  \label{lem:kernelToGap}
  Let $P$ be an irreducible Markov kernel with invariant probability measure, $\mu$. If $P$ admits a kernel density with respect to $\mu$, $K:\cX\times\cX\to \bbR$, such that $K\in L^2(\mu\otimes \mu)$,  then $\mathrm{gap}(P)>0$. 
\end{lemma}

\begin{proof}
  For compact notation, let $P_a=\frac{1}{2}(P+P^\star)$ denote the additive reversibilization of $P$. Then $P_a$ admits a kernel density with respect to $\mu$, $K_r$ defined by $K_r(x,y)=\frac{1}{2}(K(x,y)+K(y,x))$. The assumption that $K\in L^2(\mu\otimes\mu)$ implies that $K_r\in L^2(\mu\otimes\mu)$, as well.

  So, we have that
  \begin{equation}
    \label{eq:HSoperator}
    P_af(x)=\int_{\cX} \mu(dy)K_r(x,y)f(y)
  \end{equation}

  A classical result shows that if $P_a$ can be represented in the form of (\ref{eq:HSoperator}), with $\mu$ a positive measure and $K_r\in L^2(\mu\otimes\mu)$, then $P_a$ is a compact operator. See, e.g. Proposition 4.7 of \cite{conway2019course}. 

  It follows that the only possible accumulation point of $\spec(P_a)$ is at $0$. It follows that $\sup\left(\spec(P_a)\setminus\{1\}\right)<1$, and so, Lemma~\ref{lem:gapFromSpectrum} implies the result. 
\end{proof}

\begin{corollary}
  \label{cor:finiteKernelDensity}
 Let $P$ be an irreducible Markov kernel with invariant probability distribution $\mu$ over a finite set $\cX$. 
 If $\mu(\{x\})>0$ for all $x\in\cX$, then $\mathrm{gap}(P)>0$. 
\end{corollary}
\begin{proof}
  If $\mu(\{x\})>0$ for all $x\in\cX$, then $P$ admits
  a kernel density, $K(x,y)=P(x,\{y\})/\mu(\{y\})$, and $K\in L^2(\mu \otimes \mu)$ since the space is finite. 
\end{proof}

\subsection{Comparison Conditions}
\label{ss:comparison}

In this subsection, we show how $\gap(P)$ can be bounded below in terms of properties of an auxiliary Markov chain or probability measure. In particular, Lemma~\ref{lem:absGap2Gap} shows that if $\absgap(P)>0$ then $\gap(P)>0$, while Proposition~\ref{prop:doeblinGap} shows that if $P$ is uniformly geometrically ergodic, then $\gap(P)>0$.

The following is an extension of a classical comparison lemma (proved for reversible chains on a finite state space in \cite{levin2017markov}) for general Markov kernels. Closely related comparison theorems for non-reversible finite-state Markov chains  are given in \cite{chung2005laplacians}.  

\begin{lemma}
  \label{lem:comparison}
  Let $P$ and $Q$ be Markov kernels over $(\cX,\cB)$ with respective invariant probability distributions $\mu_P$ and $\mu_Q$. If
  \begin{itemize}
  \item $\mu_P \ll \mu_Q$
  \item $\mu_Q \ll \mu_P$
  \item there exists $\epsilon >0$ such that $P(x,A) \ge \epsilon Q(x,A)$ for all $x\in\cX$ and all $A\in\cB$,
  \end{itemize}
  then
  $$
  \gap(P)\ge \frac{\epsilon \gap(Q)}{ \left\|\frac{d\mu_P}{d\mu_Q} \right\|_{L^{\infty}(\mu_Q)}\left\| \frac{d\mu_{Q}}{d\mu_P}\right\|_{L^{\infty}(\mu_P)} }.  
  $$
\end{lemma}

\begin{proof}
  Let $\cE_P$ and $\cE_Q$ be the respective Dirichlet forms for $P$  and $Q$, and let $\var_P$ and $\var_Q$ be the respective variance functionals. The absolute continuity conditions imply that the Radon-Nikodym derivatives $\frac{d\mu_P}{d\mu_Q}$ and $\frac{d\mu_Q}{d\mu_P}$ exist. Furthermore, $\mu_P$ and $\mu_Q$ have the same sets of measure zero, and so $\|f\|_{L^{\infty}(\mu_Q)}=\|f\|_{L^{\infty}(\mu_P)}$ for $f:\cX\to \bbC$.  

  For measurable $v:\cX\to \bbC$:
  \begin{align*}
    \cE_{Q}(v,v)&=\frac{1}{2}\int_{\cX}\int_{\cX}\mu_Q(dx)Q(x,dy)|v(x)-v(y)|^2 \\
                &= \frac{1}{2}\int_{\cX}\int_{\cX}\mu_P(dx)\frac{d\mu_{Q}}{d\mu_P}(x)Q(x,dy)|v(x)-v(y)|^2 \\
                &\le \frac{1}{2\epsilon} \left\| \frac{d\mu_{Q}}{d\mu_P}\right\|_{L^{\infty}(\mu_P)}
                  \int_{\cX}\int_{\cX}\mu_P(dx)P(x,dy)|v(x)-v(y)|^2 \\
                &=\frac{1}{\epsilon} \left\| \frac{d\mu_{Q}}{d\mu_P}\right\|_{L^{\infty}(\mu_P)}\cE_P(v,v).
  \end{align*}

  Using that $\mu_P v =\argmin_{y\in \bbC}\bbE_P[|v(\bx)-y|^2]$, where $\bbE_P$ is the expectation with respect to $P$ gives:
  \begin{align*}
    \var_P(v) &= \int_{\cX}\mu_P(dx)|v(x) - \mu_P v|^2\\
              &\le \int_{\cX}\mu_P(dx)|v(x)-\mu_Qv|^2 \\
              &=\int_{\cX}\mu_Q(dx)\frac{d\mu_P}{d\mu_Q}(x) |v(x)-\mu_Qv|^2 \\
              &\le \left\|\frac{d\mu_P}{d\mu_Q} \right\|_{L^{\infty}(\mu_Q)} \int_{\cX}\mu_Q(dx) |v(x)-\mu_Qv|^2 \\
              &= \left\|\frac{d\mu_P}{d\mu_Q} \right\|_{L^{\infty}(\mu_Q)}\var_Q(v).
  \end{align*}
  Combining the bounds shows that if $\var_P(v)>0$, then $\var_Q(v)>0$ and 
  \begin{align*}
    \frac{\cE_P(v,v)}{\var_P(v)}\ge \frac{\epsilon}{ \left\|\frac{d\mu_P}{d\mu_Q} \right\|_{L^{\infty}(\mu_Q)}\left\| \frac{d\mu_{Q}}{d\mu_P}\right\|_{L^{\infty}(\mu_P)} } \frac{\cE_Q(v,v)}{\var_Q(v)}
  \end{align*}

  Minimizing both sides over $v$ gives the result. 
\end{proof}

In many cases, it is easier to work with the sampled chain, $P^m$, than the original chain. We can use the comparision lemma to show that if $P^m$ has a positive spectral gap, then so does $P$.

\begin{lemma}
  \label{lem:iteratedGap}
  Let $P$ be a Markov kernel with invariant probability distribution, $\mu$. Then for all $m\ge 1$:
  $$
  \gap(P)\ge 1-\left(1-\frac{1}{2^m}\gap(P^m)\right)^{1/m}.
  $$
In particular, if $\gap(P^m)>0$ for some $m\ge 1$, then $\gap(P)>0$. 
\end{lemma}
\begin{proof}
  Let $P_a = \frac{1}{2}(P+P^\star)$ be the additive reversibilization of $P$ with respect to $\mu$.  Note that $\mu$ is invariant for $P_a$, $P^m$, and $P_a^m$.

  For all $x\in\cX$ and all $A\in\cB$, we have
  \begin{equation}
    \label{eq:iteratedKernelBound}
  P_a^m(x,A) \ge  \frac{1}{2^m}P^m(x,A).
  \end{equation}

  Now, since $\mu$ is invariant for both $P^m$ and $P_a^m$, we can apply Lemma~\ref{lem:comparison} with Radon-Nikodym derivatives of $1$ to give
  $$
  \gap(P_a^m)\ge \frac{1}{2^m}\gap(P^m).
  $$

    Now we show that $1-\gap(P_a)=(1-\gap(P_a^m))^{1/m}$.

  The spectral mapping theorem shows that $\lambda\in\spec(P_a)$ if and only if $\lambda^m \in \spec(P_a^m)$. See Theorem 10.28 of \cite{rudin1991functional}. Since $P_a^m$ is self-adjoint, we have that
  \begin{align*}
    \gap(P_a^m)=1-\sup\left(\spec(P_a^m)\setminus \{1\}\right)& \implies \sup\left(\spec(P_a^m)\setminus \{1\}\right) = 1-\gap(P_a^m) \\
                                                              &\implies \sup\left(\spec(P_a)\setminus \{1\}\right)= (1-\gap(P_a^m))^{1/m}\\
    &\implies \gap(P_a)=1- (1-\gap(P_a^m))^{1/m}.
  \end{align*}

  The result now follows, since $\gap(P)=\gap(P_a)$. 
  \end{proof}

  Let $P$ be a Markov kernel with invariant probability measure, $\mu$. Let $L^2_0(\mu)$ be the subspace of functions $f\in L^2(\mu)$ such that $\mu f=0$. 
Proposition 22.2.4 of \cite{douc2018markov} shows that $\absgap(P)>0$ if and only if there is an integer, $m\ge 1$, such that $\|P^m\|_{L^2_0(\mu)}<1$. Furthermore, if $\absgap(P)>0$, then $P^m$ converges to $\ones \otimes \mu$ exponentially as an $L^2(\mu)$ operator, which implies that $\mu$ must be the unique invariant probability measure.

The next lemma shows that if $\absgap(P)>0$, then $\gap(P)>0$. 

\begin{lemma}
  \label{lem:absGap2Gap}
  Let $P$ be a Markov kernel over $(\cX,\cB)$ with invariant probability measure, $\mu$. If $\absgap(P)>0$, then there is a number $m\ge 1$ such that $\|P^m\|_{L^2_0(\mu)}<1$ and
  $$
  \gap(P)\ge 1-\left(1-2^{-m}(1-\|P^m\|_{L_0^2(\mu)}) \right)^{1/m}>0.
  $$
\end{lemma}
\begin{proof}
 
  Since $\absgap(P)>0$, there is an integer $m\ge 1$ such that $\|P^m\|_{L^2_0(\mu)}<1$.
  Based on Lemma~\ref{lem:iteratedGap}, it suffices to show that $\gap(P^m)\ge 1-\|P^m\|_{L^2_0(\mu)}$. 


  To get a lower-bound on $\gap(P^m)$, let $\cE_{P^m}$ and $\var_{P^m}$ denote the respective Dirichlet form and variance functional for $P^m$. 

  For $v\in L_0^2(\mu)$, we have that $\|v\|_{L^2(\mu)}^2=\var_{P^m}(v)$. So:
  \begin{align*}
  \Re\left(\langle v,P^m v\rangle \right)&= \Re\left(\langle v,P^m v\rangle \right) \\
  &=\Re\left(\langle v,[(P^m-I)+I]v\rangle \right) \\
                                         &=-\cE_{P^m}(v,v)+\var_{P^m}(v) \\
                                         &\le \|P^m\|_{L_0^2(\mu)} \var_{P^m}(v),
  \end{align*}
  where the final inequality uses the Cauchy-Schwartz inequality and the definition of the operator norm.

  Re-arranging shows that when $\var_{P^m}(v)>0$ and $v\in L_0^2(\mu)$, 
  \begin{equation}
    \label{eq:iteratedRatio}
  \frac{\cE_{P^m}(v,v)}{\var_{P^m}(v)}\ge 1- \|P^m\|_{L_0^2(\mu)} .
  \end{equation}

  The bound in (\ref{eq:iteratedRatio}) holds for all $v\in L^2(\mu)$ with $\var_{P^m}(v)>0$, since
  $$
\frac{\cE_{P^m}(v,v)}{\var_{P^m}(v)}=\frac{\cE_{P^m}(v-(\mu v)\ones,v-(\mu v)\ones)}{\var_{P^m}(v-(\mu v)\ones )}.
  $$
  It follows that $\gap(P^m)\ge  1- \|P^m\|_{L_0^2(\mu)}$.
\end{proof}
 
The next basic result shows that IID kernels have gap $1$:
\begin{lemma}
  \label{lem:IIDGap}
  If $P(x,A)=\mu(A)$ for all $x\in\cX$ and $A\in\cB$, then $\gap(P)=1$. 
\end{lemma}
\begin{proof}
Using (\ref{eq:dirichletAlt}) shows that $\cE(f,f)=\var(f)$ for all $f\in L^2(\mu)$, so that $\gap(P)=1$. 
\end{proof}

A Markov kernel, $P$ is called \emph{uniformly geometrically ergodic} if it has an invariant probability distribution, $\mu$ and there are constants $c\ge 0$, $\rho \in [0,1)$ such that for all $x\in\cX$ and all $n\ge 0$:
$$
\|P^n(x,\cdot)-\mu\|_{\TV}\le c\rho^n. 
$$

The Markov kernel, $P$ is uniformly geometrically ergodic if and only if there is a number $\delta>0$ and an integer $m\ge 1$ such that:
\begin{equation}
  \label{eq:doeblin}
P^m(x,A)\ge \delta \mu(A) \quad \forall x\in\cX, \forall A\in\cB.
\end{equation}
See Theorem 15.3.1 of \cite{douc2018markov}.

The next result shows that if $P$ is uniformly geometrically ergodic, then $\gap(P)>0$.

\begin{proposition}
  \label{prop:doeblinGap}
  Let $P$ and $\mu$ satisfy (\ref{eq:doeblin}) for an integer $m\ge 1$ and number $\delta > 0$. Then
  $$
  \gap(P)\ge 1-\left(1-\left(\frac{\delta}{2}\right)^m \right)^{1/m} >0. 
  $$
\end{proposition}

\begin{proof}
  Let $Q$ be the Markov kernel corresponding to independent, identically distributed samples from $\mu$. Here $Q(x,A)=\mu(A)$ for all $x\in\cX$ and all $A\in\cB$. Lemma~\ref{lem:IIDGap} implies that $\gap(Q)=1$. 

  Since $Q$ and $P^m$ both have invariant probability distribution $\mu$, Lemma~\ref{lem:comparison} implies that
  $$
  \gap(P^m) \ge \delta \gap(Q) = \delta. 
  $$
  The result now follows from Lemma~\ref{lem:iteratedGap}. 
\end{proof}

In some cases, (\ref{eq:doeblin}) can be hard to verify, especially when the form of $\mu$ is not known. The next result shows that a bound on the spectral gap can be computed as long as we can compute upper and lower bounds on the kernel in terms of another probability distribution, $\nu$. 

\begin{lemma}
  \label{lem:sandwich}
  Let $P$ be a Markov kernel and  let $\nu$ be a probability measure. If there are positive numbers $\gamma$ and $\delta$ and an integer $m\ge 1$ such that
  \begin{equation}
    \label{eq:kernelSandwich}
  \gamma \nu(A) \ge P^m(x,A)\ge \delta \nu(A) \quad \forall x\in\cX,\forall A\in\cB,
  \end{equation}
  then
  \begin{align*}
    \gap(P^m)&\ge \frac{\delta^2}{\gamma}\\
    \gap(P)&\ge 1-\left(1-2^{-m}\frac{\delta^2}{\gamma}\right)^{1/m}. 
  \end{align*}
\end{lemma}

\begin{proof}
  Let $Q(x,A)=\nu(A)$ for all $x\in\cX$ and all $A\in\cB$. Lemma~\ref{lem:IIDGap} implies that $\gap(Q)=1$.

  The right inequality of (\ref{eq:kernelSandwich}) implies that the entire state space is  a small set, and so 

  Now, since $\mu$ is invariant for $P$, we have for all $A\in\cB$:
  $$
  \gamma \mu Q(A)=\gamma \nu(A)\ge \mu P^m(A)=\mu(A) \ge \delta \mu Q(A)= \delta \nu(A). 
  $$
  It follows that
  \begin{align*}
    \left\| \frac{d\mu}{d\nu}\right\|_{L^{\infty}(\nu)}&\le \gamma \\
    \left\|\frac{d\nu}{d\mu}\right\|_{L^{\infty}(\mu)}&\le \delta^{-1}
  \end{align*}
  The result now follows from Lemmas~\ref{lem:comparison} and \ref{lem:iteratedGap}.
\end{proof}

  \subsection{Linear Gaussian Systems}
  \label{ss:linGauss}
  Now we use Lemma~\ref{lem:iteratedGap} to bound the spectral gaps of stable linear Gaussian systems. Some challenges that arise are: the state space is continuous and unbounded, the kernels are typically non-reversible, and the corresponding Markov chain is not uniformly geometrically ergodic. 

  Let $P$ be the Markov operator corresponding to the linear Gaussian  of the form:
  \begin{equation}
    \label{eq:linGauss}
  \bx_{k+1}=F\bx_k+G\bv_k
\end{equation}
  where $\bx_k\in\bbR^n$ and  and $\bv_k$ are independent, identically distributed Gaussian vectors with mean $0$ and covariance $I$.

The main result of this section, Proposition~\ref{prop:linGaussSpec}, characterizes the non-zero spectrum of $P$, gives an exact formula for the absolute spectral gap, and gives a positive lower bound on the spectral gap.

  Let $\prec$ and $\preceq$ denote the positive definite and positive semi-definite partial orders, respectively. Let $\|\cdot\|_2$ denote the spectral norm for matrices. 
  The following elementary lemma is required for the main results of the subsection. 

  \begin{lemma}
    \label{lem:elementaryLinGauss}
  Say that the system $(F,G)$ is controllable and all eigenvalues of $F$ have magnitude less than $1$. 
  \begin{enumerate}
      \item \label{it:linGaussIrreducible} $P$ is irreducible.
    \item \label{it:lyap} The unique invariant distribution for $P$, $\mu$, is a Gaussian with mean $0$ and covariance $S$, where $S\succ 0$ solves the linear Lyapunov equation:
\begin{equation}
  \label{eq:linLyap}
  S=GG^\top +FSF^\top.
\end{equation}
\item \label{it:posDefCov}
  For all integers $m$ sufficiently large (in particular, for $m\ge n$) the matrix
  $$
  W = \sum_{i=0}^{m-1}F^i G G^\top (F^i)^\top
  $$
  is positive definite.
\item \label{it:altLyap} For $m$ as above, let $A=F^m$. Then $S$ also satisfies the Lyapunov equation:
  $$
  S = W+ASA^\top.
  $$

\item \label{it:contractiveMat}  If $\tilde A =   S^{-\frac{1}{2}}F^{m}S^{\frac{1}{2}}$ and $W$ is positive definite, then $\|\tilde A\|_2 <1$.
  \end{enumerate}
\end{lemma}
\begin{proof}
  Items \ref{it:linGaussIrreducible} and \ref{it:lyap} are standard facts for linear Gaussian systems. See \cite{meyn2012markov}.

  For item~\ref{it:posDefCov}, controllability implies that the matrix
  $$
  \cC = \begin{bmatrix}G & FG & \cdots & F^{n-1}G\end{bmatrix}
  $$
  has rank $n$. Thus
  $$
  \cC\cC^\top = \sum_{i=0}^{n-1}F^i GG^\top (F^i)^\top \succ 0.
  $$
  Thus, for all sufficiently large $m\ge 0$, $W\succ 0$. 

  Item~\ref{it:altLyap} follows because for all $m\ge 0$:
  \begin{align*}
    S &= \sum_{i=0}^\infty F^i G G^\top (F^i)^\top \\
  &= \sum_{i=0}^{m-1}F^i G G^\top (F^i)^\top + F^m \left(
\sum_{i=0}^\infty F^i G G^\top (F^i)^\top 
  \right)(F^m)^\top. 
  \end{align*}

  Finally, we prove item~\ref{it:contractiveMat}.   Let $\tilde A=S^{-\frac{1}{2}}AS^{\frac{1}{2}}$  and $\tilde W = S^{-\frac{1}{2}}WS^{-\frac{1}{2}}$. Multiplying both sides of the Lyapunov equation by $S^{-\frac{1}{2}}$ gives
  $$
  I=\tilde W+\tilde A \tilde A^\top \implies
\left\|S^{-\frac{1}{2}}F^mS^{\frac{1}{2}}\right\|_2=
  \|\tilde A\|_2 =\sqrt{\lambda_{\max}(I-\tilde W)}<1,
  $$
  where $\|\tilde A\|_2$ denotes the spectral norm of $\tilde A$. 
\end{proof}

If $\cS$ is a closed subspace of $L^2(\mu)$, let $P|_{\cS}$ denote the restriction of $P$ to $\cS$ and let $\Pi_{\cS}$ denote the orthogonal projection onto $\cS$. 

For an integer $d\ge 1$, 
let $F^{\otimes d}=F\otimes \cdots \otimes F$, where the Kronecker product is repeated $d$ times. In the special case of $d=0$, we set $F^{\otimes 0}=1$.

The following lemma describes the eigenvalues of $P$ when restricted to spaces of polynomials. Note that every polynomial is in $L^2(\mu)$, since all moments of a Gaussian are bounded.  

\begin{lemma}
  \label{lem:polyRestrictions}
  Let $d$ be a non-negative integer and let $\cP_d$ be the space of polynomials of degree at most $d$. 
  \begin{enumerate}
    \item \label{it:invariance} $\cP_d$ is invariant under $P$.
    \item \label{it:eigenvalueStructure} If $\gamma$ is an eigenvalue of the finite-dimensional operator $\Pi_{\cP_d} P|_{\cP_d}$, then $\gamma$ is an eigenvalue of $F^{\otimes k}$ for some $k\in\{0,\ldots,d\}$. In particular, if $k\ge 1$ and $\lambda_1,\ldots,\lambda_n$ are the eigenvalues of $F$, then $\gamma = \prod_{j=1}^k\lambda_{i_j}$, for some numbers $i_j\in \{1,\ldots,n\}$. 
    \item \label{it:eigenvectorDegreeRelation} If $\gamma$ is an eigenvalue of $P$ with an eigenvector that is a degree $d$ polynomial, then $\gamma$ is an eigenvalue of $F^{\otimes d}$.
  \end{enumerate} 
\end{lemma}
\begin{proof}
  Item~\ref{it:invariance}:
  Let $f$ be a polynomial of degree at most $d$. We can always uniquely decompose $f$ as $f=\sum_{i=0}^{d}f_k$, where $f_k$ are homogeneous polynomials of degree $k$. Note that if $f$ has degree less than $d$, the decomposition still holds but with $f_k=0$ for $k$ greater than the degree of $f$.

  By linearity, $Pf=\sum_{k=0}^d Pf_k$. So, to show invariance, it suffices to check that if $p$ is a homogeneous polynomial of degree $k$, then $Pp$ is a polynomial of degree at most $k$. 

  Let $p$ be a homogeneous polynomial of degree $k$ and let $\bv$ be a Gaussian random variable with mean $0$ and covariance $I$. Then, for all $x\in\bbR^n$ we can express:
  \begin{equation}
  \label{eq:homInvariant}
  Pp(x)=\bbE[p(Fx+G\bv)]=p(Fx)+q(x).
\end{equation}
Here $p(Fx)$ is homogeneous of degree $k$, and $q(x)$ has degree at most $k-1$, and is formed by taking expected values of cross terms between $Fx$ and $G\bv$. In the special case of $k=0$, \eqref{eq:homInvariant} still holds, since $p(x)=c$ for some constant $c$. Then \eqref{eq:homInvariant} reduces to $Pc=c$, with $q=0$. Thus, $\cP_d$ is invariant under $P$ for all $d\ge 0$.

  Item \ref{it:eigenvalueStructure}: For $f\in\cP_d$, let $g=Pf$. Decompose $f=\sum_{k=0}^d f_k$ and $g=\sum_{k=0}^{d}g_k$ into sums of homogeneous polynomials. Since $\cP_k$ is invariant under $P$ for all $k$, $\Pi_{\cP_d}P|_{\cP_d}$ can be expressed in block-triangular form:
  \begin{equation}
  \label{eq:triangularPolyMap}
  \begin{bmatrix}
  g_0 \\
  g_1 \\
  \vdots \\
  g_d
  \end{bmatrix}=\begin{bmatrix}
  P_{00} & P_{01} & \cdots  & P_{0d}\\
  & P_{11} & \cdots & P_{1d}\\
 & & \ddots &\vdots \\
         &&  & P_{dd}
  \end{bmatrix}\begin{bmatrix}
  f_{0}\\
  f_1\\
  \vdots\\
  f_d
\end{bmatrix}.
  \end{equation}

  The block triangular structure implies that if $\gamma$ is an eigenvalue of $\Pi_{\cP_d}P|_{\cP_d}$, then $\gamma$ must be an eigenvalue of $P_{kk}$ for some $k\in\{0,\ldots,d\}$. For $k=0$, irreducibility of $P$ implies that $P_{00}=1$, and so the only  eigenvalue is $1$, which is precisely the eigenvalue of $F^{\otimes 0}$. 

  For $k\ge 1$, and $p$ homogeneous of degree $k$,  $P_{kk}p$ is the homogeneous part of degree $k$ of $Pp$. By \eqref{eq:homInvariant}, $P_{kk}p(x)=p(Fx)$. It now remains to prove that every eigenvalue of $P_{kk}$ is an eigenvalue of $F^{\otimes k}$.

  In what follows, we will show that $P_{kk}$ can be expressed as a quotient of a mapping between tensors, and we will derive the eigenvalues of $P_{kk}$ from the eigenvalues of the tensor mapping. To this end, we will follow common indexing conventions for tensor calculations. The components of $x\in \bbR^n$ are denoted by $x^i$. The components of the matrix $F$ are denoted as $F_i^j$ so that 
  $$
  (Fx)^j = \sum_{i=1}^n F_i^j x^i.
  $$

If $p$ is any homogeneous polynomial of degree $k$, there are coefficients $\alpha_{i_1,\ldots,i_k}$ such that 
$$
p(x)=
\sum_{i_1,i_2,\ldots,i_k=1}^{n}\alpha_{i_1,i_2,\ldots,i_k} x^{i_1} x^{i_2}\cdots x^{i_k}
$$

Following \cite{lee2013introduction}, let $T^k((\bbC^n)^\star)$ denote the space of $\mathrm{rank}-k$ covariant tensors. The numbers $\alpha_{i_1,\ldots,i_k}$ can be identified with the coefficients of a $\mathrm{rank}-k$ covariant tensor, $\alpha:(\bbC^n)^{k}\to \bbC$ defined by:
$$
\alpha(v_1,\ldots,v_k)=\sum_{i_1,\ldots,i_k=1}^{n} \alpha_{i_1,\ldots,i_k}v_1^{i_1}\cdots v_k^{i_k}.
$$

  Note that the polynomial $p$ can be expressed as $p(x)=\alpha(x,\ldots,x)$, but the representation via tensors is typically not unique. For example, if $n=2$ and $k=2$, we have
  \begin{align*}
    \alpha(x,x)&=\alpha_{11}x^1 x^1 +\alpha_{12} x^1 x^2 + \alpha_{21} x^2 x^1 +\alpha_{22} x^2x^2 \\
               &= \alpha_{11}(x^1)^2+(\alpha_{12}+\alpha_{21})x^1x^2+\alpha_{22}(x^2)^2.
  \end{align*}
  So, if $u\in T^{2}((\bbC^2)^\star)$ with $u^{11}=u^{22}=0$ and $u^{12}=-u^{21}$, we will have $u(x,x)=0$ for all $x\in\bbR^2$ and so
  $$
  h(x)=\alpha(x,x)=(\alpha+u)(x,x)
  $$
  for all $x$.

  More generally, let $\cU=\{u\in T^k((\bbC^n)^\star)| u(x,\ldots,x)=0 \: \forall x\in\bbR^n\}$. We see that if $\alpha$ and $\beta$ are two rank-$k$ covariant tensors, then $\alpha(x,\ldots,x)$ and $\beta(x,\ldots,x)$ define the same polynomial if and only if $\alpha-\beta \in\cU$. This means that the homogeneous polynomials of degree $k$ can be identified with the elements of the quotient space $T^k((\bbC^n)^\star)/\cU$. As discussed in \cite{axler2024linear}, elements of the quotient space are equivalence classes of the form $\alpha +\cU$. (Here $\alpha+\cU=\beta+\cU$ if and only if $\alpha-\beta \in\cU$, which holds if and only if $\alpha$ and $\beta$ define the same polynomial.)

  The mapping $p(x)\mapsto p(F x)$ can be extended to a linear mapping $L:T^{k}((\bbC^n)^\star)\to T^k((\bbC^n)^\star)$ by
  $$
  L(\alpha)(v_1,\ldots,v_k)=\alpha(F v_1,\ldots,Fv_k).
  $$
  Note that if $\alpha \in\cU$, then $L(\alpha)\in\cU$, so that $\cU$ is invariant under $L$. As in \cite{axler2024linear}, the quotient map $L/\cU$ is defined by
  $$
  (L/\cU)(\alpha+\cU):=L(\alpha)+\cU. 
  $$
  In particular, say that $p(x)=\alpha(x,\ldots,x)$. Then 
  for any $u\in\cU$, we have that
  \begin{align*}
    (L(\alpha)+u)(x,\ldots,x)&=L(\alpha)(x,\ldots,x)\\
                             &=\alpha(F x,\ldots,F x)\\
                             &=p(F x).
  \end{align*}
  It follows that for all $u\in\cU$, $L(\alpha)+u$ defines the same polynomial as $P_{kk}p$. Thus, when identifying polynomials with tensor quotients, the action of $P_{kk}$  is precisely $L/\cU$. In particular, $P_{kk}$ and $L/\cU$ must have the same eigenvalues.

  Now, we recall a general property of quotient mappings. If $L$ is a linear operator on a finite dimensional space, $\cU$ is invariant under $L$, and $\gamma $ is an eigenvalue of $L/\cU$, then $\gamma$ is  also an eigenvalue of $L$. See Chapter 5, Exercise 38(b) in \cite{axler2024linear}.

To calculate the eigenvalues of $L$, note that 
\begin{align*}
  L(\alpha)(v_1,\ldots,v_k)&=\sum_{i_1,\ldots,i_k=1}^n\sum_{j_1,\ldots,j_k=1}^n\alpha_{i_1,\ldots,i_k}(F_{j_1}^{i_1}v_1^{j_1})\cdots (F_{j_k}^{i_k} v_k^{j_k}) \\
                           &=\sum_{i_1,\ldots,i_k=1}^n\left(\sum_{j_1,\ldots,j_k=1}^m \alpha_{j_1,\ldots,j_k} F_{i_1}^{j_1} \cdots F_{i_k}^{j_k}\right)v_1^{i_1}\cdots v_k^{i_k} \\
                           & =\sum_{i_1,\ldots,i_k=1}^n\left(\sum_{j_1,\ldots,j_k=1}^m \alpha_{j_1,\ldots,j_k} (F^{\otimes k})_{i_1,\ldots,i_k}^{j_1,\ldots,j_k}\right)v_1^{i_1}\cdots v_k^{i_k} 
\end{align*}
  In particular, if $\alpha$ is interpreted as an $n^k$-dimensional row vector, $L$ has the matrix representation
  $$
  L(\alpha)=\alpha F^{\otimes k}.
  $$
  It follows that the set of eigenvalues of $L$, which contains the set of eigenvalues of $P_{kk}$, are precisely the eigenvalues of $F^{\otimes k}$. The specific form of an eigenvalue follows because the eigenvalues of $F^{\otimes k}$ are precisely the $k$-fold products of eigenvalues of $F$. 

  Item~\ref{it:eigenvectorDegreeRelation}: If an eigenvector of $\gamma$ has degree $d$, then $\gamma$ must be an eigenvalue of $P_{dd}$. It then follows that $\gamma$ is an eigenvalue of $F^{\otimes d}$. 
\end{proof}

The following is a discrete-time version of  results from \cite{metafune2002spectrum}. It generalizes the results on linear Gaussian systems from \cite{diaconis2008gibbs,khare2009rates}, which handle the case of reversible Markov operators induced by linear Gaussian systems.

  \begin{proposition}
  \label{prop:linGaussSpec}
  Say that all eigenvalues of $F$ have magnitude less and $1$ and  the pair $(F,G)$ is controllable. Let $\lambda_1,\ldots,\lambda_n$ denote the eigenvalues of $F$ (which may not be distinct). The following statements hold:
  \begin{enumerate}
    \item \label{it:linGaussEvals} If $\gamma \in\spec(P)\setminus\{0,1\}$ then $\gamma = \prod_{j=1}^k\lambda_{i_j}$ for some sequence of numbers $i_j\in \{1,\ldots,n\}$.
    \item \label{it:linGaussAbsGap} $\absgap(P)\ge 1-\rho(F)$ where $\rho(F)$ is the spectral radius of $F$.
    \item \label{it:linGaussGap} Let $\tilde A$ and $m$ be as in Lemma~\ref{lem:elementaryLinGauss}. Then 
      $$ 
      \gap(P)\ge 1-\left(1-2^{-m}\left(1-\|\tilde A\|_2\right)\right)^{1/m} >0. 
      $$ 
  \end{enumerate}
  \end{proposition}
  \begin{proof}
    Item~\ref{it:linGaussEvals}: Let $m\ge 1$, $S$, $A$, and $W$ be as in Lemma~\ref{lem:elementaryLinGauss}. By the spectral mapping theorem $\spec(P^m)=\spec(P)^m$, and so it suffices to prove the statement for $Q:=P^m$ and then utilize that $\lambda$ is an eigenvalue of $A$ if and only if $\lambda^{1/m}$ is an eigenvalue of $F$.

The invariant density can be written explicitly as
  \begin{equation}
    \label{eq:gaussianDensity}
  \mu(dy)=\frac{1}{(2\pi)^{\frac{n}{2}}|S|^{\frac{1}{2}}}\exp\left(-\frac{1}{2}y^\top S^{-1}y \right)dy=:p_S(y)dy.
\end{equation}

The iterated kernel, $Q=P^m$ has the form:
  $$
  Q(x,dy)=P^m(x,dy)=\frac{1}{(2\pi)^{\frac{n}{2}}|W|^{\frac{1}{2}}}\exp\left(-\frac{1}{2}(y-Ax)^\top W^{-1}(y-Ax) \right)dy.
  $$
  In other words, the iterated system is  a special case of the linear Gaussian system, \eqref{eq:linGauss}, with transition matrix, $A$, and noise covariance $W$. Furthermore, $A$ is stable and the system $(A,W^{1/2})$ is controllable, so that Lemmas~\ref{lem:elementaryLinGauss} and \ref{lem:polyRestrictions} can be applied to $Q$ as well. 

  The iterated kernel,  $Q$, admits a density with respect to  $\mu$ given by:
  $$
  K(x,y)=\frac{|S|^{\frac{1}{2}}}{|W|^{\frac{1}{2}}}\exp\left(
    -\frac{1}{2}(y-Ax)^\top W^{-1}(y-Ax) +\frac{1}{2}y^\top S^{-1} y
  \right).
  $$

  We will show that $K\in L^2(\mu\otimes \mu)$, and thus $Q$ is a Hilbert-Schmidt operator, which implies that $Q$ is compact.

  Direct calculation gives
  \begin{align*}
    \MoveEqLeft[0]
    \int_{\cX\times\cX}K(x,y)^2 \mu(dx)\mu(dy)=
    \frac{1}{(2\pi)^n|W|}\int_{\bbR^n\times \bbR^n}
    \exp\left (-
  \begin{bmatrix}
  x\\
  y
  \end{bmatrix}^\top R \begin{bmatrix}
x \\
y
\end{bmatrix}\right)dx dy,
  \end{align*}
  where
  $$
  R = \begin{bmatrix}
  \frac{1}{2}S^{-1}+A^\top W ^{-1}A & -A^\top W ^{-1} \\
  -W^{-1}A & W^{-1}-\frac{1}{2}S^{-1}
  \end{bmatrix}.
  $$
  It follows that $K\in L^2(\mu\otimes \mu)$ as long as $R \succ 0$, where $\succ$ denotes the positive definite partial order.

  Note that $S\succeq W$ and so $W^{-1}\succeq S^{-1}$. Thus, the diagonal blocks are both positive definite. Using Schur complements, the Woodbury matrix identity, and some algebra:
  \begin{align*}
    R\succ 0 &\iff \frac{1}{2}S^{-1}+A^\top W ^{-1}A- A^\top W^{-1}\left( W^{-1}-\frac{1}{2}S^{-1}\right)^{-1} W^{-1}A \succ 0\\
             &\iff \frac{1}{2}S^{-1}+A^\top \left(W^{-1}-W^{-1}(W+W(2S-W)^{-1}W)W^{-1} \right)A \succ 0 \\
             &\iff \frac{1}{2}S^{-1}-A^\top (2S-W)^{-1}A \succ 0 \\
             &\iff \begin{bmatrix}
               \frac{1}{2}S^{-1} & A^\top \\
               A & 2S-W
             \end{bmatrix} \succ 0 \\
             &\iff 2S-W - 2A S A^\top \succ 0 \\
             &\iff W \succ 0 .
  \end{align*}
  Since $W$ is positive definite, we have that $R$ is positive definite, and so $K\in L^2(\mu\otimes\mu)$. 

  So, $Q$ is a compact operator, and thus the spectrum is countable and all non-zero elements of the spectrum are eigenvalues. See, for example, Theorem 4.25 of \cite{rudin1991functional}.

  Since $A=F^m$ and $\lambda_i$ are the eigenvalues of $F$, we have that the eigenvalues of $A$ are $\lambda_1^{m},\ldots,\lambda_n^m$. If $\gamma $ is an eigenvalue of $Q$ with a polynomial eigenvector, Lemma~\ref{lem:polyRestrictions} implies that either $\gamma = 1$ or $\gamma = \prod_{j=1}^{k}\lambda_{i_j}^{m}$ for some sequence of numbers $i_j\in\{1,\ldots,n\}$. So, to prove item~\ref{it:linGaussEvals} for $Q$, it suffices to show that every non-zero eigenvalue of $Q$ has a polynomial eigenvector. In this case, the corresponding element of $\spec(P)$ will be $\gamma^{1/m}$, which will have the claimed form.  

  Recall that $\cP_d$ is the set of polynomials of degree at most $d$. Let $\cC_d$ be the orthogonal complement of $\cP_d$ in $L^2(\mu)$.

For a fixed value of $d$, we can uniquely decompose $f\in L^2(\mu)$ as $f=f_1+f_2$, where:
\begin{equation}
  \label{eq:2dSplit}
\begin{bmatrix}
  f_1 \\
  f_2
\end{bmatrix}
=
\begin{bmatrix}
  \Pi_{\cP_d}f \\
  \Pi_{\cC_d}f
\end{bmatrix}.
\end{equation}
Note that in this case, $f_1$ is a polynomial of degree at most $d$. 

Applying the same decomposition to $g=Qf$, we have that $g=g_1+g_2$, where
\begin{equation}
  \label{eq:triangularDecomposion}
  \begin{bmatrix}
    g_1 \\
    g_2
    \end{bmatrix} = \begin{bmatrix}
  \left(\Pi_{\cP_d} Q|_{\cP_d}\right) &   \left(\Pi_{ \cP_d} Q|_{\cC_d}\right) \\
    0 &   \left(\Pi_{\cC_d} Q|_{\cC_d}\right)
  \end{bmatrix}
\begin{bmatrix}
  f_1 \\
  f_2
\end{bmatrix}.
\end{equation}
The block triangular structure arises because the invariance property from Lemma~\ref{lem:polyRestrictions} implies that $\Pi_{\cC_d}Q|_{\cP_d}=0$. 

Recall the matrix $\tilde A = S^{-\frac{1}{2}}F^m S^{\frac{1}{2}}$. Lemma~\ref{lem:elementaryLinGauss} showed that $\|\tilde A\|_2<1$. We will show that the induced norm of $\Pi_{\cC_d}Q|_{\cC_d}$ satisfies:
\begin{equation}
  \label{eq:operatorNormDecay}
  \|\Pi_{\cC_d}Q|_{\cC_d}\|_{L^2(\mu)}\le \|\tilde A\|_2^{d+1}. 
\end{equation}
Thus, for any $\gamma \ne 0$, if $\|\tilde A\|_2^{d+1} < |\gamma|$, then $\gamma I - \Pi_{\cC_d}Q|_{\cC_d}$ has a bounded inverse in $\cC_d$. 

Now, say that $\gamma\ne 0$ is an eigenvalue of $Q$. Pick $d$ such that $\|\tilde A\|_2^{d+1} < |\gamma|$. The eigenvalue equation can be expressed as:
$$
  \begin{bmatrix}
    0 \\
    0
    \end{bmatrix} = \begin{bmatrix}
  \left(\gamma I - \Pi_{\cP_d} Q|_{\cP_d}\right) &   \left(-\Pi_{ \cP_d} Q|_{\cC_d}\right) \\
    0 &   \left(\gamma I - \Pi_{\cC_d} Q|_{\cC_d}\right)
  \end{bmatrix}
\begin{bmatrix}
  f_1 \\
  f_2
\end{bmatrix}.
$$
Assuming that \eqref{eq:operatorNormDecay} holds, invertibility of $\gamma I -  \Pi_{\cC}Q|_{\cC_d}$ implies that $f_2=0$. Thus, any eigenvector corresponding to $\gamma$ must be  a polynomial of degree at most $d$. 

So, to complete the proof of item~\ref{it:linGaussEvals}, it suffices to prove \eqref{eq:operatorNormDecay}. Note that non-expansiveness of projections implies that $\|\Pi_{\cC_d}Q|_{\cC_d}\|_{L^2(\mu)}\le \|Q|_{\cC_d}\|_{L^2(\mu)}$, so it suffices to bound $\|Q|_{\cC_d}\|_{L^2(\mu)}$. 

Let $Q^\star$ denote the adjoint kernel for $Q$, and let $M=Q^\star Q$, which is known as the multiplicative reversibilization of $Q$. Note that for $f\in L^2(\mu)$, 
$$
\|Qf\|_{L^2(\mu)}^2 = \langle f, Mf\rangle.
$$

The adjoint kernel, $Q^\star$ can be expressed as   $Q^\star(x,dy)=K(y,x)\mu(dy)$, which, after some algebra, is given by
  $$
  Q^\star(x,dy)=\frac{1}{(2\pi)^{\frac{n}{2}}|\Omega|^{\frac{1}{2}}}\exp\left(-\frac{1}{2}(y-A^\star x)^\top \Omega^{-1}(y-A^\star x) \right)dy
  $$
  where
  \begin{align*}
    \Omega &=(S^{-1}+A^\top W^{-1} A)^{-1}\\
    A^\star&=\Omega A^\top W^{-1}=SA^\top S^{-1}.
  \end{align*}

  Here $A^\star$ is the adjoint of $A$ with respect to the inner product $\langle x,y\rangle = x^\top S^{-1} y$. 

  Thus, we have that $Q^\star$ encodes a linear Gaussian system with transition matrix $A^\star$ and covariance $\Omega$. It follows that  $Q^\star Q = M$ encodes a special case of the linear Gaussian system, \eqref{eq:linGauss}, with transition matrix $A^\star A$ and covariance $\Gamma:=A^\star W(A^\star)^\top+\Omega\succ 0$. Note that $A^{\star} A=S^{\frac{1}{2}}\tilde A^\top \tilde A S^{-\frac{1}{2}}$, so that $A^\star A$ and $\tilde A^\top \tilde A$ have the same eigenvalues. Then, since $\|\tilde A\|_2<1$, the eigenvalues of $A^\star A$ must be bounded in magnitude by $\|\tilde A\|_2^2 <1$. In particular, Lemmas~\ref{lem:elementaryLinGauss} and \ref{lem:polyRestrictions} hold for $M$. 

  Lemma~\ref{lem:polyRestrictions} implies that $\cP_d $ is invariant with respect to $M$ for all $d\ge 0$.
 Furthermore, $M$ is self-adjoint. Thus, the finite-dimensional operator $\Pi_{\cP_d} M|_{\cP_d}$ is also self-adjoint. Thus, if $g_d$ is the dimension of $\cP_d$, the space $\cP_d$ must have an orthonormal basis $\{\psi_1,\ldots,\psi_{g_d}\}$, such that $\psi_i$ are all eigenvectors of $M$. 

  For any $d\ge 0$, let $\cH_{d+1}$ be the set of polynomials of degree $d+1$ which are orthogonal to all polynomials of degree at most $d$. Then  
  $$
  \cP_{d+1}= \cP_{d}\oplus \cH_{d+1}.
  $$

If $\{\phi_{g_d+1},\ldots,\phi_{g_{d+1}}\}$ forms an orthonormal basis of $\cH_{d+1}$, then \\
$\{\psi_1,\ldots,\psi_{g_d},\phi_{g_{d}+1},\ldots,\phi_{g_{d+1}}\}$ forms an orthonormal basis of $\cP_{d+1}$. The matrix of $\Pi_{\cP_{d+1}} M|_{\cP_{d+1}}$ with respect to this basis must have the form:
$$
\begin{bmatrix}
  D & 0 \\
  0 & H
\end{bmatrix}
$$
where $D$ is a $g_d\times g_d$ diagonal matrix and $H$ is a Hermitian matrix. The eigenvectors of $H$ can be used to construct eigenvectors of $M$, $\psi_{g_d+1},\ldots,\psi_{g_{d+1}} \in\cH_{d+1}$, such that $\{\psi_1,\ldots,\psi_{g_{d+1}}\}$ is an orthonormal basis of $\cP_{d+1}$. Continuing in this fashion, we can inductively construct bases of $\cP_d$ for all orders $d\ge 0$.

For notational consistency, set $\cH_0=\cP_0$, which is the space of constants. Then, the orthogonal subspaces, $\cH_i$, define the classical polynmial chaos expansion, and so
$$
L^2(\mu)=\bigoplus_{i=0}^{\infty}\cH_i.
$$
See Theorem 2.6 of \cite{janson1997gaussian}.

If we set $g_{0}=1$, $g_{-1}=0$, and $\psi_{1}=1$, then $\cH_d=\mathrm{span}\{\psi_{g_{d-1}+1},\ldots,\psi_{g_d}\}$ for all $i\ge 0$.

Each $\psi_i$ is an eigenvector of $M$, so for all $i\ge 1$, let $\zeta_i$ be the eigenvalue corresponding to $\psi_i$. Then, for $f\in L^2(\mu)$, the eigenvalue decomposition implies
$$
Mf=\sum_{i=1}^{\infty}\zeta_i \langle \psi_i,f\rangle \psi_i.
$$

Furthermore, for each $i$, $\psi_i\in\cH_d$ for some $d$. It follows from Lemma~\ref{lem:polyRestrictions} that $\zeta_i$ is an eigenvalue of $(A^\star A)^{\otimes d}$. As discussed above, all of the eigenvalues of $A^\star A$ have magnitude bounded by $\|\tilde A\|_2^2 <1$. It follows that $|\zeta_i|\le \|\tilde A\|_2^{2d}$.

If  $f\in\cC_d$, then there are coefficients $\beta_i$ for $i>g_d$ such that 
$$
f=\sum_{i>g_d}\beta_i \psi_i \quad \textrm{and} \quad \|f\|_{L^2(\mu)}^2=\sum_{i>g_d}|\beta_i|^2.
$$ 
Furthermore, if $f\in\cC_d$, then:
$$
  \langle \psi_i,f\rangle =  \begin{cases}\beta_i & i> g_d \\
    0 & i\le g_d \\
  \end{cases}.
$$

It follows that 
\begin{align*}
  \|Qf\|_{L^2(\mu)}^2 &= \langle f,M f\rangle \\
           &=\sum_{i=1}^{\infty}\zeta_i \left|\langle f,\psi_i\rangle\right|^2  \\
           &=\sum_{i>g_d}\zeta_i |\beta_i|^2 \\
           &\le \|f\|_{L^2(\mu)}^2 \max\{\zeta_i|i>g_d\} \\
           &\le\|f\|_{L^2(\mu)}^2 \|\tilde A\|_2^{2(d+1)}.
\end{align*}
Thus~\eqref{eq:operatorNormDecay} holds and the proof of item~\ref{it:linGaussEvals} is complete. 

Item~\ref{it:linGaussAbsGap} is now an immediate consequence of item~\ref{it:linGaussEvals}.

For item~\ref{it:linGaussGap}, note that  $L^2_0(\mu)$ is precisely $\cC_0$, which is an invariant subspace for $P^m=Q$. Thus 
$$
\|\Pi_{\cC_0}Q|_{\cC_0}\|_{L^2(\mu)} = \|P^m\|_{L_0^2(\mu)}.
$$
It follows from \eqref{eq:operatorNormDecay} that $\|P^{m}\|_{L_0^2(\mu)}\le \|\tilde A\|_2 < 1$. The result now follows from Lemma~\ref{lem:absGap2Gap}.
\end{proof}

\section{A $V$-Uniformly Geometrically Ergodic Markov Chain with No Spectral Gap}
\label{sec:noGap}

In this subsection, we construct a Markov kernel, $P$, over a countable state space such that $P$ is $V$-uniformly geometrically ergodic, but $\gap(P)=0$. The Markov kernel in this subsection is a simplified version of the Markov kernel from \cite{haggstrom2005central}. The Markov kernel from \cite{haggstrom2005central} was used in \cite{kontoyiannis2012geometric} to give an example of a Markov kernel which is $V$-uniformly geometrically ergodic, but has $\absgap(P)=0$. By Example~\ref{ex:noGap2AbsGap}, having $\absgap(P)=0$ does not necessarily imply that $\gap(P)=0$. In contrast, by Lemma~\ref{lem:absGap2Gap}, if we can show that $\gap(P)=0$, then $\absgap(P)=0$ as well. In particular, the next result extends Theorem 1.4 of \cite{kontoyiannis2012geometric} and gives an alternative proof.

\begin{proposition}
  There is a Markov chain on a countable state space which is $V$-uniformly geometrically ergodic, but has $\gap(P)=\absgap(P)=0$. 
\end{proposition}

\begin{proof}
   
The state space for the Markov chain is
$$
\cX=\{0\}\cup \bigcup_{a=1}^\infty\{(a,1),(a,2),\ldots,(a,a)\}.
$$

For simpler notation, we identify $P(x,\{y\})$ with $P(x,y)$ and identify $\mu(\{x\})$ with $\mu(x)$ when $x,y\in \cX$. The Markov kernel is given by:
$$
P(x,y)=\begin{cases}
  \frac{1}{2} & x=y=0 \\
  2^{-a-1} & x=0, y=(a,a) \\
  1 & x=(a,b), y=(a,b-1) \textrm{ with } 1<b\le a\\
  1 & x=(a,1), y=0 \textrm{ with } a\ge 1 \\
  0 & \textrm{otherwise}
\end{cases}
$$
and its invariant probability measure is
$$
\mu(x)=\begin{cases}
  \frac{1}{2} & x=0 \\
  2^{-a-2} & x=(a,b) \textrm{ with } 1\le b\le a.
\end{cases}
$$

Here $P$ is irreducible and aperiodic, and $0$ is a recurrent atom. In particular, $\{0\}$ is a petite set. We will show that for any $\rho \in (1,\sqrt{2})$ that $P$ is $V$-uniformly geometrically ergodic with
$$
V(x)=\begin{cases}
  \rho^b & x=(a,b) \textrm{ with } 1\le b \le a \\
  1 & x=0.
\end{cases}
$$

Direct calculation shows that $P$ and $V$ satisfy the drift inequality:
$$
PV(x)\le \rho^{-1} V(x)+\left(\frac{1}{2}+\frac{1}{2}\frac{\rho}{2-\rho}\right)\indic_{\{0\}}(x) \textrm{ for all } x\in\cX.
$$
Furthermore,
$$
\|V\|_{L^2(\mu)}^2 = \frac{1}{2}+\frac{1}{2}\frac{\rho^2}{2-\rho^2}.
$$

So, $P$ is $V$-uniformly geometrically ergodic for some $V\in L^2(\mu)$. 

Now we turn to showing that $P_a=\frac{1}{2}(P+P^\star)$ has a sequence of eigenvalues (with respect to $L^2(\mu)$) which converges to $1$. 

Since $\mu(y)>0$ for all $y\in\cX$, we have that $P$ admits a kernel density given by $K(x,y)=P(x,y)/\mu(y)$. The corresponding adjoint kernel is given by $P^\star(x,y)=K(y,x)\mu(y):$
$$
P^\star(x,y)=\begin{cases}
  \frac{1}{2} & x=y=0 \\
  2^{-a-1} & x=0, y=(a,1) \\
  1 & x=(a,b), y=(a,b+1) \textrm{ with } 1\le b < a \\
  1 & x=(a,a), y=0 \textrm{ with } a\ge 1 \\
  0 & \textrm{otherwise}.
\end{cases}
$$

If we order the states like $0,(1,1),(2,1),(2,2),\ldots,(n,1),\ldots,(n,n),\ldots$, then $P_a=\frac{1}{2}(P+P^\star)$ can be represented by the infinite matrix:
$$
P_a=\begin{bmatrix}
  \frac{1}{2} & \frac{1}{4} & \frac{1}{2^{2+2}}c_2^\top & \frac{1}{2^{3+2}}c_3^\top & \cdots & \frac{1}{2^{n+2}}c_n^\top &  \cdots \\
  1 & 0 & 0 & 0 & \cdots & 0 & \cdots\\
  \frac{1}{2}c_2 & 0 & \frac{1}{2}M_2  & 0 & \cdots & 0 & \cdots \\
  \frac{1}{2}c_3 & 0 & 0 & \frac{1}{2}M_3 &  & 0 &  \cdots \\
  \vdots & \vdots & \vdots & &  \ddots & \vdots  \\
  \frac{1}{2}c_n & 0 & 0 & & & \frac{1}{2}M_n \\
  \vdots & \vdots & \vdots & & & & \ddots
\end{bmatrix}
$$
where
\begin{align*}
  c_n^\top & = \begin{bmatrix}
    1 &
    0 &
    \cdots &
    0 &
    1
  \end{bmatrix} \in\mathbb{R}^{1\times n} \\
  M_n&= \begin{bmatrix}
    0 & 1 & 0 \\
    1 & 0 & 1 \\
    0 & 1 & 0 & \\
    & & & \ddots \\
    & & & & 0 & 1 \\
    & & &  & 1 & 0
  \end{bmatrix} \in\bbR^{n\times n} .
\end{align*}

So, if $v$ is an eigenvector of $M_n$ such that $\lambda v = M_n v$ and $c_n^\top v=0$, then $\tilde v=\begin{bmatrix}0 \\ v \\ 0\end{bmatrix}$ (with appropriately sized zeros) is an eigenvector of $P_a$ with $P_a \tilde v = \frac{1}{2}\lambda \tilde v$. If $\frac{\lambda}{2} \ne 1$, then we must have $\mu \tilde v=0$, since
$$
\mu P_a \tilde v = \mu \tilde v = \frac{\lambda}{2}\mu \tilde v \implies \left(1-\frac{\lambda}{2}\right)\mu \tilde v = 0. 
$$
Furthermore, since $\tilde v$ has only finitely many non-zero values, it must have $\|\tilde v\|_{L^2(\mu)}<\infty$.

We will show, in particular, that for all $k\ge 2$, there is an eigenvector $v_k$ of $M_{2^{k}-1}$ such that $M_{2^{k}-1}v_k=\lambda_{k-2}v_k$ and $c_{2^k-1}^\top v_k=0$, where $\lambda_{k-2}$  is defined by the recursion:
\begin{align*}
  \lambda_0&=0 \\
  \lambda_{i+1}&=\sqrt{2+\lambda_i}.
\end{align*}

If $0\le \lambda_i < 2$, the recursion satisfies $\lambda_i < \lambda_{i+1}< 2$. So, $\lambda_i$ must converge to a fixed point of the recursion in the interval $[0,2]$, and the only such fixed point is $2$. Thus, the corresponding sequence of eigenvectors of $P_a$, denoted by $\tilde v_k = \begin{bmatrix}0\\ v_k \\ 0\end{bmatrix}$, satisfies
\begin{align*}
  \lim_{k\to\infty}
  \frac{\cE_P(\tilde v_k,\tilde v_k)}{\var(\tilde v_k)}&=\lim_{k\to\infty}\frac{\langle \tilde v_k,(I-P_a)\tilde v_k \rangle }{\|\tilde v_k\|_{L^2(\mu)}^2}\\
                                                       &=\lim_{k\to\infty}\left(1-\frac{\lambda_{k-2}}{2}\right) \\
                                                       &=0. 
\end{align*}
So, if such eigenvalues and eigenvectors exist, we must have $\gap(P)=0$, and then Lemma~\ref{lem:absGap2Gap} implies $\absgap(P)=0$ as well. 

To construct the required eigenvectors, we will derive several properties of the eigenvectors of $M_n$ for $n\ge 1$. For consistency of the analysis, we set $M_1=0\in\bbR$.

Define a sequence of polynomials, $q_i$ by:
\begin{subequations}
\label{eq:charPolyRecursion}
\begin{align}
  q_0(s)&=1 \\
  q_1(s)&=s \\
  q_{i}(s)&=sq_{i-1}(s)-q_{i-2}(s) \textrm{ for } i\ge 2.
\end{align}
\end{subequations}
Throughout this proof, $s$ will denote the general variable for the polynomials constructed. 

Manipulation of the equation $M_n v = \lambda v$ shows that the characteristic polynomial of $M_n$ is $q_n$ and
$v$ is an eigenvector of $M_n$ if and only if it has the form
\begin{equation}
  \label{eq:eigenvectorForm}
v = z\begin{bmatrix}
  q_0(\lambda) \\
  q_1(\lambda) \\
  \vdots \\
  q_{n-1}(\lambda)
\end{bmatrix}
\end{equation}
where $\lambda$ is an eigenvalue of $M_n$ and $z$ is a non-zero scalar. Note that $q_0(\lambda)=1$, so that $v\ne 0$ as long as $z\ne 0$. In particular, when choosing an eigenvector, we can set $z=1$. 

To find the desired eigenvalues and eigenvectors, we will also utilize the following auxiliary sequence of polynomials:
\begin{align*}
  r_0(s)&=s \\
  r_{i+1}(s)&=r_i(s)^2-2 \textrm{ for } i\ge 0.
\end{align*} 

The existence of the desired eigenvalue/eigenvector pairs is proved from the following claims:
\begin{enumerate}
\item For $k\ge 1$,  $q_{2^k-1}(s)=\prod_{i=0}^{k-1}r_i(s)$ \label{claim:factor}
\item \label{claim:roots}
  The roots of $r_i$ are given by $\hat \lambda_{i,0},\cdots,\hat \lambda_{i,2^i-1}$ which are defined recursively by:
  \begin{align*}
    \hat \lambda_{0,0}&=0 \\
    \hat \lambda_{i+1,2j}&=\sqrt{2+\hat \lambda_{i,j}} \textrm{ for } j=0,\ldots,2^i-1 \\
    \hat \lambda_{i+1,2j+1}&=-\sqrt{2+\hat \lambda_{i,j}} \textrm{ for } j=0,\ldots,2^i -1
  \end{align*}
\item
  \label{claim:opposites}
  If $q_n(\lambda)=0$, then $q_{n+i}(\lambda)=-q_{n-i}(\lambda)$ for $i=0,\ldots,n$. 
\end{enumerate}

Before proving the claims, we  show how they can be used to get the desired eigenvalue/eigenvector pair: Claim \ref{claim:factor} implies that any root of $r_{k-2}$ must also be a root of $q_{2^{k-1}-1}$ and $q_{2^k-1}$. Comparing the recursions for $\lambda_i$ and $\hat \lambda_{i,j}$ shows that $\lambda_i=\hat\lambda_{i,0}$ for $i\ge 0$. In particular, $\lambda_{k-2}$ is a root of $r_{k-2}$. Then $q_{2^k-1}(\lambda_{k-2})=0$ implies that $\lambda_{k-2}$ is an eigenvalue of $M_{2^{k}-1}$. Using that $q_{2^{k-1}-1}(\lambda_{k-2})=0$, the form of the eigenvectors in (\ref{eq:eigenvectorForm}), and Claim \ref{claim:opposites} shows that an eigenvector of $M_{2^k-1}$ corresponding to $\lambda_{k-2}$ is given by
\begin{equation*}
  v_k=\begin{bmatrix}
    q_0(\lambda_{k-2}) \\
    \vdots \\
    q_{2^{k-1}-2}(\lambda_{k-2}) \\
    0 \\
    -q_{2^{k-1}-2}(\lambda_{k-2}) \\
    \vdots \\
    -q_0(\lambda_{k-2})
  \end{bmatrix}.
\end{equation*}
In particular, we have $c_{2^k-1}^\top v_k=0$ and $M_{2^k-1}v_k=\lambda_{k-2}v_k$. So, $\frac{\lambda_{k-2}}{2}$ must be an eigenvalue of $P_a$. 

Now we prove Claim~\ref{claim:factor}. At $k=1$, the formula holds since:
$$
q_1(s)=r_0(s)=s.
$$
To prove the claim, assume inductively that the formula from the claim holds at some $k\ge 1$. It  then suffices to show that
\begin{equation}
  \label{eq:qSplit}
  q_{2^{k+1}-1}(s)=q_{2^{k}-1}(s)r_{k}(s).
\end{equation}

To prove (\ref{eq:qSplit}), we will need to prove a few intermediate formulas. The first is that for all $0< i<n$:
\begin{align}
  \label{eq:qLowHigh}
q_n(s)=q_i(s)q_{n-i}(s)-q_{i-1}(s)q_{n-i-1}(s)
\end{align}
The case of $i=1<n$ follows from (\ref{eq:charPolyRecursion}). Now, assume that (\ref{eq:qLowHigh}) holds for some $0<i<n-1$. Using the recursive definition from (\ref{eq:charPolyRecursion}) twice gives:
\begin{align*}
  q_n(s)&=q_i(s)\left(sq_{n-i-1}(s)-q_{n-i-2}(s)\right)-q_{i-1}(s)q_{n-i-1}(s)\\
        &=\left(sq_i(s)-q_{i-1}(s)\right)q_{n-i-1}(s)-q_i(s)q_{n-i-2}(s) \\
        &= q_{i+1}(s)q_{n-i-1}(s)-q_i(s)q_{n-i-2}(s).
\end{align*}
Thus, (\ref{eq:qLowHigh}) holds for all $1<i<n$ by induction. 

If for $k\ge 1$, $n=2^{k+1}-1$, and $i=2^{k}-1$, we have $n-i=2^{k}$, $i-1=2^{k}-2$, and $n-i-1=2^{k}-1$, so that (\ref{eq:qLowHigh}) can be written as:
$$
q_{2^{k+1}-1}(s)=q_{2^{k}-1}(s)\left(q_{2^{k}}(s)-q_{2^{k}-2}(s) \right).
$$

Thus, to prove (\ref{eq:qSplit}), it suffices to show that
\begin{equation}
  \label{eq:tailIdentity}
q_{2^k}(s)-q_{2^{k}-2}(s)=r_k(s) \textrm{ for } k\ge 1.
\end{equation}

The proof of (\ref{eq:tailIdentity}) will rely on the intermediate identity:
\begin{equation}
  \label{eq:squareIdentity}
  q_i(s)^2=q_{i+1}(s) q_{i-1}(s)+1 \textrm{ for } i\ge 1.
\end{equation}
At $i=1$:
\begin{align*}
  q_{2}(s)q_0(s)+1&=(s^2-1)\cdot (1)+1 \\
                  &=s^2 \\
                  &=q_1(s)^2. 
\end{align*}
Assume inductively that (\ref{eq:squareIdentity}) holds for some $i\ge 1$. Using (\ref{eq:charPolyRecursion}) twice, followed by the induction assumption gives:
\begin{align*}
  q_{i+2}(s)q_{i}(s)+1&=\left(s q_{i+1}(s)-q_{i}(s)\right)q_i(s)+1 \\
                      &=q_{i+1}(s)\left( sq_i(s)\right) - q_i(s)^2+1 \\
                      &=q_{i+1}(s)\left(q_{i+1}(s)+q_{i-1}(s)\right)-q_i(s)^2+1\\
                      &=q_{i+1}(s)^2+\left(q_{i+1}(s)q_{i-1}(s)+1\right) - q_i(s)^2 \\
                      &=q_{i+1}(s)^2.
\end{align*}
Thus, (\ref{eq:squareIdentity}) holds by induction. 

Now we prove (\ref{eq:tailIdentity}). 
At $k=1$, we have
\begin{align*}
  q_2(s)-q_0(s)&=(s^2 -1) - 1 \\
               &=r_1(s).
\end{align*}

Now assume inductively that (\ref{eq:tailIdentity}) holds at some $k\ge 1$. Applying (\ref{eq:qLowHigh}) at $n=2^{k+1}$ and $i=2^k$, and then at $n=2^{k+1}-2$ and $i=2^k-1$ gives
\begin{align*}
  q_{2^{k+1}}(s)&=q_{2^k}(s)^2-q_{2^k-1}(s)^2 \\
  q_{2^{k+1}-2}(s)&=q_{2^{k}-1}(s)^2-q_{2^{k}-2}(s)^2
\end{align*}
so that
$$
q_{2^{k+1}}(s)-q_{2^{k+1}-2}(s)=q_{2^k}(s)^2-2q_{2^k-1}(s)^2 +q_{2^{k}-2}(s)^2.
$$
Now using (\ref{eq:squareIdentity}) and the induction hypothesis gives
\begin{align*}
  q_{2^{k+1}}(s)-q_{2^{k+1}-2}(s)&= q_{2^k}(s)^2-2\left(q_{2^k}(s)q_{2^k-2}(s)+1 \right) +q_{2^{k}-2}(s)^2 \\
                                 &=\left(q_{2^k}(s)-q_{2^{k}-2}(s) \right)^2 - 2 \\
                                 &=r_k(s)^2 - 2 \\
                                 &=r_{k+1}(s). 
\end{align*}
Thus, (\ref{eq:tailIdentity}) has been proved by induction and so Claim \ref{claim:factor} holds.

Now we prove Claim~\ref{claim:roots}. The proof relies on the following identity for $r_i$:
\begin{equation}
  \label{eq:rComposition}
  r_n(s)=r_i(r_{n-i}(s)) \textrm{ for } 0\le i \le n.
\end{equation}
Note that for all $n\ge 0$, $r_n(s)=r_0(r_n(s))=r_n(r_0(s))$ since $r_0(s)=s$, while for all $n\ge 1$, $r_n(s)=r_1(r_{n-1}(s))$ since $r_1(s)=s^2-2$. In particular, (\ref{eq:rComposition}) holds at $n=0$ and $n=1$.

Now, inductively assume that for some $n\ge 1$,  $r_{i}(s)=r_j(r_{i-j}(s))$ for all $0\le j\le i\le n$. As discussed above, 
$$
r_{n+1}(s)=r_{n+1}(r_0(s))=r_{0}(r_{n+1}(s))=r_1(r_n(s)).
$$
Now fix $1\le i\le n$. Then 
\begin{align*}
  r_{n+1}(s)=r_1(r_n(s))=r_1(r_{i-1}(r_{n-(i-1)}(s)))=r_{i}(r_{n+1-i}(s))=r_{i}(r_{n+1-i}(s)).
\end{align*}
Thus, we have shown that \eqref{eq:rComposition} holds for all $n\ge 0$.

In particular, (\ref{eq:rComposition}) implies that the recursion for $r_i$ can be expressed alternatively as:
$$
r_{i+1}(s)=r_i(r_1(s))=r_i(s^2-2).  
$$

Now, $r_0(s)=s$, so its only root is $\hat \lambda_{0,0}=0$. Inductively assume that the roots of $r_i$ are given by $\hat\lambda_{i,0},\ldots\hat\lambda_{i,2^i-1}$. Since $r_i$ is a monic polynomial we can express $r_{i+1}$ as:
\begin{align*}
  r_{i+1}(s)&=r_i(s^2-2) \\
            &=\prod_{j=0}^{2^{i}-1}(s^2-2-\hat\lambda_{i,j})\\
            &=\prod_{j=0}^{2^{i}-1}\left(\left(s-\sqrt{2+\hat\lambda_{i,j}}\right)\left(s+\sqrt{2+\hat\lambda_{i,j}}\right) \right)
\end{align*}
Thus, the recursive formula for the roots  of $r_i$ holds and Claim~\ref{claim:roots} is proved. 

Finally, we prove Claim~\ref{claim:opposites}. If $q_n(\lambda)=0$ and $i=0$, then $q_{n+i}(\lambda)=-q_{n-i}(\lambda)=0$. If $1\le n$, then using (\ref{eq:charPolyRecursion}) and that $q_n(\lambda)=0$ gives:
\begin{align*}
  q_{n+1}(\lambda)&=\lambda q_n(\lambda)-q_{n-1}(\lambda) \\
                  &=-q_{n-1}(\lambda). 
\end{align*}
Fix $0\le i <n$ and assume inductively that $q_{n+j}(\lambda)=-q_{n-j}(\lambda)$ for all $j$ with $0\le j\le i$. Then, using (\ref{eq:charPolyRecursion}), the induction hypothesis, and then (\ref{eq:charPolyRecursion}) again gives:
\begin{align*}
  q_{n+i+1}(\lambda)&=\lambda q_{n+i}(\lambda)-q_{n+i-1}(\lambda) \\
                    &=-\left(\lambda q_{n-i}(\lambda)-q_{n-i+1}(\lambda)\right) \\
                    &=-\left((q_{n-i-1}(\lambda)+q_{n-i+1}(\lambda))-q_{n-i+1}(\lambda)\right) \\
                    &=-q_{n-i-1}(\lambda). 
\end{align*}
Thus, Claim~\ref{claim:opposites} has been proved by induction, and the proof of the proposition is complete. 
\end{proof}

\section{Conclusion}
\label{sec:conclusion}
This paper gives methods to compute bounds on (Dirichlet) spectral gaps for non-reversible Markov chains, motivated by analysis of temporal difference algorithms. The current results give general approaches to computing spectral gap bounds and specific bounds for the case of linear Gaussian systems.  
Further bounds for spectral gaps of Markov chains are certainly possible. For example, \cite{chung2005laplacians} gives bounds on the spectral gap for finite-state Markov chains via a variation on Cheeger's inequality. Future work will examine the interplay between stability for nonlinear stochastic systems and spectral gap bounds, as well as continuous-time variations.

\begin{funding}
  The author was supported in part by:  NSF ECCS 2412435 and NIH \\
1R01NS147767-0.
\end{funding}



\bibliographystyle{imsart-number} 
\bibliography{CoOL-bib/cool-refs}       

\end{document}